\documentclass[10pt]{article}
\usepackage{amsmath,amsfonts,euscript,amssymb,amsthm,graphicx,verbatim, esint}
\usepackage{color}
\usepackage{hyperref}
\usepackage{enumerate}
\usepackage{mathrsfs}

\usepackage{mathtools}
\usepackage{caption}
\usepackage{subcaption}
\usepackage{authblk}

\newtheorem{theorem}{Theorem}[section]
\newtheorem{proposition}[theorem]{Proposition}
\newtheorem{lemma}[theorem]{Lemma}
\newtheorem{corollary}[theorem]{Corollary}

\theoremstyle{remark}
\newtheorem{remark}{Remark}

\newcommand{\R}{\mathbb{R}}
\newcommand{\N}{\mathbb{N}}
\newcommand{\Z}{\mathbb{Z}}
\newcommand{\cA}{\mathcal{A}}

\newcommand{\cF}{\mathcal{F}}
\newcommand{\cG}{\mathcal{G}}
\newcommand{\cR}{\mathcal{R}}
\newcommand{\St}{\mathbb{S}^2}

\newcommand{\eps}{\varepsilon}
\newcommand{\ein}{\boldsymbol{e}}
\newcommand{\bs}[1]{ \boldsymbol{#1}}

\newcommand{\m}{\boldsymbol{m}}
\renewcommand{\ni}{\boldsymbol{n}^{\mathrm{in}}}
\newcommand{\no}{\boldsymbol{n}^{\mathrm{out}}}
\newcommand{\n}{\boldsymbol{n}}
\newcommand{\bu}{\boldsymbol{u}}
\newcommand{\bv}{\boldsymbol{v}}
\newcommand{\bw}{\boldsymbol{w}}
\newcommand{\dif}{\ensuremath{\,\mathrm{d}}}

\DeclarePairedDelimiter{\abs}{\lvert}{\rvert} 
\DeclarePairedDelimiter{\norm}{\lVert}{\rVert} 
\DeclarePairedDelimiter{\bra}{(}{)} 
\DeclarePairedDelimiter{\pra}{[}{]} 
\DeclarePairedDelimiter{\set}{\{}{\}} 

\title{Existence and strong-field asymptotics of skyrmions in a fourth-order model of frustrated ferromagnets}

\author{Xinye Li}

\AtEndDocument{%
 \par
 \medskip
 \begin{tabular}{@{}l@{}}%
School of Mathematics and Statistics, Central South University,\\ 410083 Changsha, China\\
 \textit{Email:} \texttt{xinye.li@csu.edu.cn}\\

 \end{tabular}}

\begin{document}
\maketitle

\begin{abstract}
We study a fourth-order variational model for two-dimensional frustrated ferromagnets with competing exchange interactions and an applied magnetic field of strength $H>0$. For every $H>1/4$, we prove that the energy admits minimizers in the topological classes $Q=\pm1$ and that every minimizing sequence is precompact in $H^2$ modulo translations. The main difficulty is that spectral coercivity degenerates as $H\downarrow1/4$. Using a Helmholtz circle-mean identity, we prove that the residual energy of every nonzero-degree configuration has a uniform positive lower bound, even at the degenerate endpoint. Together with a sphere-valued $H^2$-splitting construction, this provides a threshold-stable binding inequality and yields compactness throughout the coercive regime. We also identify $H=1/4$ as the sharp spectral threshold.  Below it the energy is unbounded from below, whereas at the threshold nonzero-degree configurations retain a positive energy barrier and degree-zero Weyl sequences lose compactness. Finally, in the strong-field regime, rescaled minimizers approach those minimizers of the limiting functional that maximize the Dirichlet energy, while topological-charge and normalized-energy measures concentrate on the $H^{-1/4}$-scale.
\end{abstract}



\section{Introduction and main results}

Magnetic skyrmions are localized two-dimensional magnetization configurations carrying a nontrivial topological degree. 
Much of the rigorous mathematical literature concerns chiral magnets in which the Dzyaloshinskii--Moriya interaction (DMI) favors twisting, while exchange, anisotropy, and Zeeman terms determine the energetic cost and spatial scale of localized skyrmion textures \cite{Melcher2014,DoringMelcher2017,BernandMantelMuratovSimon2021}. A distinct stabilization mechanism arises in inversion-symmetric frustrated magnets, where competing exchange interactions give rise, at the continuum level, to higher-order exchange energies. Isolated skyrmions and continuum descriptions based on this mechanism were developed in \cite{Leonov2015,LinHayami2016,Kharkov2017}. Three-dimensional knotted extensions were investigated in \cite{Sutcliffe2017}. Experimentally, skyrmion phases associated with exchange frustration have been observed in several centrosymmetric materials \cite{Hirschberger2019, Kurumaji2019}. 

Motivated by these developments, we undertake a variational analysis of skyrmions in frustrated ferromagnetic models. More specifically, we work in two spatial dimensions with magnetization fields $\m=(m_1,m_2,m_3):\R^2\to\St$. After renormalization and a rotation in magnetization space aligning the external field with the $\ein_3$-direction, the energy takes the form
\begin{equation}\label{eq:EH-intro}
E_H(\m)=\int_{\R^2} \bra*{-\frac{1}{2} |\nabla \m|^2+\frac{1}{2} |\Delta \m|^2 + H(1-m_3)} \dif x,
\end{equation}
where $H>0$ denotes the applied-field strength.
The negative second-order term energetically favors spatial variations of the magnetization, while the positive fourth-order term controls high-frequency oscillations. Their competition provides the continuum mechanism for frustration-induced stabilization. The Zeeman term $H(1-m_3)$ selects the uniform state $\ein_3$ as the far-field vacuum. Accordingly, the natural finite-energy class is
\[
\cA:=\set*{\m:\R^2 \to \St, \, \m-\ein_3 \in H^2(\R^2;\R^3)}.
\]
By the Sobolev embedding $H^2(\R^2) \hookrightarrow C^{0,\gamma}(\R^2)$ for every $\gamma \in (0,1)$, each $\m\in\cA$ admits a uniformly continuous representative. Since $\m-\ein_3\in H^2(\R^2;\R^3)$, this representative satisfies $\m(x)\to\ein_3$ as $|x|\to\infty$. Consequently, setting $\m(\infty)=\ein_3$ defines a continuous map from the one-point compactification $\R^2 \cup \set{\infty} \cong \St$ into $\St$, whose topological degree is given by
\[
Q(\m)=\frac{1}{4\pi}\int_{\R^2} q(\m) \, \dif x, \quad q(\m)=\m \cdot \partial_1 \m \times \partial_2 \m.
\]
Our main object is the variational problem
\[
I_q(H):=\inf \set*{E_H(\m), \, \m \in \cA_q}, \quad\text{where } \cA_q=\set*{\m \in \cA, \, Q(\m)=q}.
\]
The energy $E_H$ is invariant under spatial reflections, whereas an orientation-reversing reflection changes the sign of $Q$. Consequently, it suffices to work in the class $Q=-1$; all the results below have identical counterparts in the class $Q=+1$.

Rigorous variational theories for skyrmion-type models have established existence and stability on the whole plane \cite{Melcher2014,LiMelcher2018}, as well as existence, compactness, and asymptotic concentration in confined geometries \cite{MonteilMuratovSimonSlastikov2023,MuratovSimonSlastikov2025}. For the fourth-order frustrated model considered here, Harland established topological lower bounds proportional to the  topological degree \cite{Harland2019}, while Speight and Winyard analyzed the low-field regime, where the uniform state is replaced by a conical spiral background \cite{SpeightWinyard2020}.  Related deterministic and stochastic evolution equations with competing second- and fourth-order exchange have been studied in \cite{DoresicMelcher2022,GoldysJiaoMelcher2025}.  However, these works do not establish existence and compactness of minimizers in the degree-$\pm1$ classes throughout $H>1/4$, or their strong-field asymptotics.

A first structural feature of the variational problem is the sharp threshold $H=1/4$. Indeed, writing $\bu=\m-\ein_3$ and using $|\m|=1$, we have
\[
1-m_3=\frac{1}{2}|\m-\ein_3|^2=\frac{1}{2}|\bu|^2.
\]
Hence
\[
E_H(\m)
=
\frac12\int_{\R^2} |\Delta \bu|^2-|\nabla \bu|^2+H| \bu|^2 \dif x .
\]
By Plancherel's theorem,
\[
E_H(\m) =\frac{1}{2}\int_{\R^2} \bra*{|\xi|^4-|\xi|^2+H} |\widehat \bu(\xi)|^2 \dif \xi 
=\frac{1}{2}\int_{\R^2} \bra*{\left(|\xi|^2-\frac12\right)^2+H-\frac{1}{4}} |\widehat \bu(\xi)|^2 \dif \xi.
\]
Thus, $E_H$ is coercive for every $H>1/4$, i.e. there exists $c_H>0$ such that
\[
E_H(\m) \ge c_H \norm{\m-\ein_3}^2_{H^2(\R^2)}.
\]

The characteristic length scale follows from Derrick's scaling argument \cite{Derrick1964}. Set $\m_R(x)=\m(x/R)$ for $R>0$. Then $Q(\m_R)=Q(\m)$, while
\[
E_H(\m_R) = -\frac{1}{2} \int_{\R^2} |\nabla \m|^2 \dif x + \frac{1}{2R^2} \int_{\R^2} |\Delta \m|^2 \dif x + HR^2 \int_{\R^2} 1-m_3 \dif x.
\]
For a localized profile of characteristic size $R$, the fourth-order exchange and field terms are of order $1/R^2$ and $HR^2$, respectively, whereas the Dirichlet term is scaling invariant. Balancing the two scale-dependent positive contributions gives $R \sim H^{-1/4}$. It is useful to formulate the problem on this natural length scale. For $0 \le \eps \le 2$, set
\[
\cF_\eps(\n):=\frac{1}{2} \int_{\R^2} |\Delta \n|^2 +|\n - \ein_3|^2 -\eps |\nabla \n|^2 \dif x
\]
and define
\[
\widehat{I}_q(\eps):= \inf\set*{\cF_\eps(\n), \, \n \in \cA, \, Q(\n)=q}.
\]
For $H \ge1/4$, let $\eps = H^{-1/2}$. The scaling $\m(x)=\n(H^{1/4}x)$ gives
\begin{equation}\label{eq:exact}
E_H(\m)=\sqrt{H} \cF_\eps(\n), \qquad I_q(H)=\sqrt{H} \widehat{I}_q(\eps).
\end{equation}
Thus the finite-field and rescaled problems are equivalent. The rescaled formulation also includes the limiting problem $\eps=0$ and turns the strong-field regime $H \to \infty$ into the parameter limit $\eps \to 0$.

We first establish existence and compactness for the rescaled problem.
\begin{theorem}\label{thm:existence}
For every $0 \le \eps <2$, the infimum $\widehat{I}_{-1}(\eps)$ is attained, i.e. there exists $\n_\eps \in \cA_{-1}$ such that
\[
\cF_\eps(\n_\eps) = \widehat{I}_{-1}(\eps).
\]
Moreover, if $(\n_k)_{k \in \N} \subset \cA_{-1}$ is a minimizing sequence for $\widehat{I}_{-1}(\eps)$, then there exist translations $a_k \in \R^2$ and a subsequence along which 
\[
\n_k (a_k + \cdot) - \ein_3 \longrightarrow \n_*-\ein_3 \quad\text{strongly in }H^2(\R^2;\R^3)
\]
for a minimizer $\n_* \in \cA_{-1}$.

In particular, if $0<\eps<2$, $H=\eps^{-2}$, and $\m_H(x)=\n_\eps(H^{1/4}x)$, then $\m_H$ minimizes $E_H$ over $\cA_{-1}$. Hence $I_{-1}(H)$ is attained for every $H>1/4$.
\end{theorem}

The principal novelty behind Theorem~\ref{thm:existence} is a topological energy barrier that remains uniform up to $\eps=2$. Harland \cite{Harland2019} proved a topological lower bound, in the rescaled variables,
\[
\cF_\eps(\n) \ge \frac{16 \pi}{3} \bra*{2-\eps}|Q(\n)|,
\]
which is linear in degree and positive throughout the coercive regime, but degenerates as $\eps \uparrow 2$, and therefore cannot provide a threshold-stable splitting barrier. The key to our proof  is the exact residual decomposition
\[
\cF_\eps(\n)=\frac{1}{2} \norm{(\Delta+1)(\n-\ein_3)}^2_{L^2(\R^2)} + \bra*{1-\frac{\eps}{2}}\norm{\nabla \n}^2_{L^2(\R^2)}.
\]
If $Q(\n)\ne0$, then $\n$ is surjective and hence attains the south pole. Combining this fact with a forced Helmholtz circle-mean identity gives
\[
\frac{1}{2} \norm{(\Delta+1)(\n-\ein_3)}^2_{L^2(\R^2)} \ge  b_*>0
\]
where $r_*$ denotes the first positive zero of the Bessel function $J_1$ and
\[
b_* := \frac{8\pi J_0(r_*)^2}{1-J_0(r_*)^2}>0
\]
is an explicit constant independent of $\eps$.  Together with the classical topological lower bound \cite{BelavinPolyakov1975}, this yields
\[
\cF_\eps(\n) \ge b_*+8\pi\bra*{1-\frac{\eps}{2}}|Q(\n)|, \quad\text{if}\quad Q(\n)\ne 0.
\]
Thus every nonzero-degree component pays the fixed cost $b_*$, even at the degenerate endpoint. Consequently, a splitting into two nonzero-degree components costs at least $2b_*$ at $\eps=2$. An explicit degree-$-1$ trial map lies strictly below this barrier, and a comparison of the $\eps$-dependence propagates the resulting binding gap throughout $0\le\eps\le 2$.

The $H^1$-decompositions used in baby-Skyrme, Faddeev and chiral skyrmion models \cite{LinYang2004, LinYang2007, Melcher2014} do not provide the $H^2$-estimates needed in the concentration-compactness argument \cite{Lions1984} in the present fourth-order setting. We therefore develop a sphere-valued $H^2$ cutoff construction that yields exact degree splitting and asymptotic energy splitting. Combined with the threshold-stable topological energy barrier, this construction provides the binding inequality needed to rule out dichotomy.

The residual formulation also reveals a sharp topological distinction at the spectral threshold $H=1/4$: nonzero-degree configurations retain a positive residual energy barrier, whereas degree-zero configurations admit spreading Weyl sequences with vanishing energy.

\begin{theorem}\label{thm:threshold}
The following statements hold.
\begin{enumerate}[(i)]
\item If $0<H<1/4$, then $I_q(H)=-\infty$ for every $q \in \Z$.
\item At $H=1/4$, if $Q(\m) \neq 0$, then $E_{1/4}(\m)\ge  b_*/2$.
\item There exists a family $ (\bw_R)_{R \gg 1}\subset \cA_0$, indexed by a spreading scale $R$, such that
\[
 \liminf_{R\to\infty}\norm{\bw_R-\ein_3}_{H^2}>0,
 \qquad E_{1/4}(\bw_R)=O(R^{-2}),
 \]
 and, for every fixed radius $A>0$,
 \[
 \sup_{y\in\R^2}\int_{B_A(y)}
\bra*{ |\bw_R-\ein_3|^2+|\nabla \bw_R|^2+|D^2 \bw_R|^2 }\dif x \to 0
\quad\text{as }R\to\infty.
 \]

 \item The function $H \mapsto I_{-1}(H)$ is continuous from the coercive side at the threshold:
 \[
 \lim_{H\downarrow1/4}I_{-1}(H)
 =I_{-1}(1/4)\ge b_*/2>0.
 \]
\end{enumerate}
\end{theorem}

\begin{remark}
\begin{enumerate}[(1)]
\item The equality in Theorem~\ref{thm:threshold} (iv) concerns only the value of the infimum and does not imply attainment of $I_{-1}(1/4)$. Moreover, part (iii) exhibits a loss of $H^2$ compactness at the threshold that cannot be restored by translations. This loss of compactness does not by itself decide whether the endpoint infimum is attained, and the present paper does not address this question. 
\item For $0<H<1/4$, the uniform state is no longer the appropriate ground state. The natural low-field formulation instead uses a conical spiral background together with a correspondingly renormalized energy \cite{SpeightWinyard2020}. Note that Theorem~\ref{thm:threshold}(i) should not be interpreted as a nonexistence statement for skyrmions relative to the low-field background. 
\end{enumerate}
\end{remark}

The preceding theorem identifies the sharp lower boundary of the uniform-vacuum coercive regime. 
At the opposite end of this regime, Theorem~\ref{thm:existence} guarantees the existence of degree-$-1$ minimizers for every $H>1/4$, making it natural to investigate their energy, size, and profiles as $H\to\infty$, equivalently $\eps=H^{-1/2} \downarrow 0$. Recalling \eqref{eq:exact}, we write
\[
\cF_\eps(\n) = \cF_0(\n) - \eps D(\n)
\qquad\text{where }
D(\n) =\frac{1}{2}\int_{\R^2}|\nabla \n|^2 \dif x.
\]
The limiting variational problem at $\eps=0$ and the set of its minimizers are
\[
\widehat{I}_{-1}(0) = \inf \set*{\cF_0(\n), \, \n \in \cA_{-1}}, \qquad \cG=\set*{\n \in \cA_{-1}:\, \cF_0(\n)=\widehat{I}_{-1}(0)}.
\]
By Theorem~\ref{thm:existence} with $\eps=0$, the set $\cG$ is nonempty and compact in $H^2$ modulo translations. 
The following theorem establishes strong-field profile selection and the first-order energy expansion.
\begin{theorem}\label{thm:strong-field}
The supremum defining
\[
D_*:= \sup\set*{D(\n), \, \n \in \cG}
\]
is attained, and the selected set $\cG_*:=\set*{\n \in \cG: \, D(\n)=D_*}$ is nonempty. If $\m_H$ is a minimizer of $I_{-1}(H)$, then
\[
\inf_{a \in \R^2, \n \in \cG_*} \norm{\m_H(a + H^{-1/4} \cdot )- \n}_{H^2} \to 0 \quad\text{as}\quad H \to \infty,
\]
and
\[
I_{-1}(H) = \sqrt{H} \widehat{I}_{-1}(0) - D_* + o(1).
\]
\end{theorem}
Thus the leading-order strong-field profiles are governed by the limiting functional $\cF_0$, rather than by the Dirichlet energy. This differs from the conformal strong-field limits of DMI models, in which the Dirichlet energy determines the leading-order behavior and shrinking harmonic maps, hence Belavin--Polyakov profiles, emerge in the limit \cite{DoringMelcher2017,BernandMantelMuratovSimon2021, KomineasMelcherVenakides2020,GustafsonWang2021}. 

Beyond profile selection, quantized concentration of topological charge densities has also been established for continuum and lattice baby Skyrme type energies in a distinct strong-anisotropy regime \cite{Briani2026}. 
In the present model, returning to the original spatial scale, the strong profile convergence, combined with a Pohozaev identity, yields concentration of the topological charge and energy and identifies the characteristic core scale $H^{-1/4}$.

\begin{theorem}\label{thm:concentration}
Let $H_k\to\infty$ and let $\m_{H_k}$ be a minimizer of $I_{-1}(H_k)$. After passing to a subsequence, there exist centers $a_k\in\R^2$ and a profile $\n_*\in\cG_*$ such that
\[
 \m_{H_k}(a_k+H_k^{-1/4}\,\cdot)-\ein_3  \longrightarrow \n_*-\ein_3  \quad\text{strongly in }H^2(\R^2).
\]
With $\mathscr{T}_a \nu = (x \mapsto x-a)_{\#}\nu$, we have
\[
\mathscr{T}_{a_k} \bra*{q(\m_{H_k})\dif x} \stackrel{*}{\rightharpoonup} -4\pi \delta_0.
\]
Moreover, the fourth-order and potential energy contributions concentrate separately:
\begin{align*}
\mathscr{T}_{a_k}\bra*{\frac{|\Delta \m_{H_k}|^2}{2\sqrt{H_k}}\dif x}
 &\stackrel{*}{\rightharpoonup}\frac{\widehat{I}_{-1}(0)}{2}\delta_0,\\
\mathscr{T}_{a_k}\bra*{\sqrt{H_k}(1-m_{H_k,3})\dif x }
 &\stackrel{*}{\rightharpoonup} \frac{\widehat{I}_{-1}(0)}{2}\delta_0.
\end{align*}
The full energy density satisfies
\[
\mathscr{T}_{a_k}\bra*{ \frac{1}{\sqrt{H_k}}\pra*{\frac{1}{2} |\Delta\m_{H_k}|^2 -\frac{1}{2}|\nabla \m_{H_k}|^2 + H_k(1-m_{H_k,3})}\dif x }
 \stackrel{*}{\rightharpoonup} \widehat{I}_{-1}(0) \delta_0.
\]

For $\lambda\in(0,1)$, define the potential-energy concentration radius at level $\lambda$ by
\[
 R_{k,\lambda} :=
 \inf\set*{r>0:  \int_{B_r(a_k)}(1-m_{H_k,3})\dif x  \ge  \lambda\int_{\R^2}(1-m_{H_k,3})\dif x }.
\]
Then there exist $0<c_\lambda<C_\lambda<\infty$ such that
\[
 c_\lambda H_k^{-1/4}
 \le R_{k,\lambda}
 \le C_\lambda H_k^{-1/4}
\]
for all sufficiently large $k$.
\end{theorem}

The paper is organized as follows.  In Section \ref{sec:existence-proof}, we prove existence and compactness of minimizers by concentration-compactness, using a sphere-valued $H^2$ decomposition and a threshold-stable binding inequality.  In Section \ref{sec:threshold}, we analyze the coercivity threshold.  In Section \ref{sec:strong-field}, we develop the strong-field asymptotic theory, including the variational limit, profile selection, Pohozaev balance, concentration measures, and the core-scale estimate.

\section{Existence via concentration-compactness}\label{sec:existence-proof}

We prove Theorem~\ref{thm:existence} by concentration-compactness. The proof has three main ingredients. First, the residual decomposition yields coercivity locally uniformly for $\eps\in[0,2)$ and reduces the argument to excluding vanishing and dichotomy. Second, a sphere-valued $H^2$ cutoff construction gives exact degree splitting and asymptotic energy splitting. Third, a uniform positive binding gap excludes splittings into two nonzero-degree components, while coercivity eliminates nontrivial degree-zero remainders. These ingredients yield compactness modulo translations and attainment of the infimum.

Expanding the square and integrating by parts, we obtain
\begin{equation}\label{eq:Feps-residual}
\cF_\eps(\n)=\frac{1}{2} \norm{(\Delta+1)(\n-\ein_3)}^2_{L^2(\R^2)} + \bra*{1-\frac{\eps}{2}}\norm{\nabla \n}^2_{L^2(\R^2)}
\end{equation}
Moreover, Young's inequality gives
\begin{equation}\label{eq:interpolation}
\norm{\nabla \n}^2_{L^2(\R^2)} 
\le \norm{\n-\ein_3}_{L^2} \norm{\Delta \n}_{L^2} \le \frac{1}{2} \bra*{\norm{\n-\ein_3}_{L^2}^2 + \norm{\Delta \n}_{L^2}^2} = \cF_0(\n).
\end{equation}
Hence,
\begin{equation}\label{eq:F0-comparison}
\bra*{1-\frac{\eps}{2}} \cF_0(\n) \le \cF_\eps(\n) \le \cF_0(\n), \quad 0 \le \eps \le 2.
\end{equation}
In particular, $\cF_\eps$ is uniformly $H^2$-coercive for $\eps$ ranging over any compact subset of $[0,2)$. 
Fix $\eps \in [0,2)$ and let $(\n_k)_{k \in \N}\subset \cA_{-1}$ be a minimizing sequence for $\widehat{I}_{-1}(\eps)$. By coercivity \eqref{eq:F0-comparison}, $(\n_k-\ein_3)$ is bounded in $H^2(\R^2;\R^3)$. Hence, after passing to a subsequence, there exists some $\n_* \in \cA$ such that $\n_k-\ein_3\rightharpoonup \n_*-\ein_3$ weakly in $H^2$. Local Rellich compactness shows that the sphere constraint passes to weak $H^2$ limits. Weak lower semicontinuity of $\cF_\eps$ follows from the residual representation \eqref{eq:Feps-residual}, and thus
\[
\cF_\eps(\n_*) \le \liminf_{k \to \infty} \cF_\eps(\n_k) = \widehat{I}_{-1}(\eps).
\]
Consequently, the direct method would be complete once one proves that $Q(\n_*)=-1$. The possible loss of topology under weak convergence is addressed by the concentration-compactness argument.

For the minimizing sequence $(\n_k)_{k \in \N} \subset \cA_{-1}$, we consider the $H^2$-mass density associated with $\n_k - \ein_3$
\begin{equation}\label{eq:rho-density}
\rho_k = |\n_k - \ein_3|^2 + |\nabla \n_k|^2 + |D^2 \n_k|^2 \in L^1(\R^2).
\end{equation}
After passing to a further subsequence, the total mass converges:
\[
\int_{\R^2} \rho_k(x) \dif x \longrightarrow M,
\]
for some constant $M> 0$. Indeed, the classical topological lower bound gives
\[
 \int_{\R^2} |\nabla \n_k|^2 \dif x \ge 8\pi |Q(\n_k)| = 8\pi,
\]
and therefore $M \ge 8\pi>0$. We write $B_R(x)=\set{y \in \R^2, \, |x-y| <R}$ for the open disk of radius $R>0$ centered at $x \in \R^2$, and $B_R=B_R(0)$. Then, according to Lions' concentration-compactness lemma \cite{Lions1984}, one of the following three situations must occur for a subsequence of $(\rho_k)_{k \in \N}$: 

\begin{enumerate}[(a)]
\item Compactness: There exists a sequence $(x_k)_{k \in \N} \subset \R^2$ such that, for every $\eta>0$, there is a radius $R>0$ such that
\[
\int_{B_R(x_k)} \rho_k(x) \dif x \ge M- \eta
\]
for all sufficiently large $k$.
\item Vanishing: For every $R>0$
\[
\lim_{k \to \infty} \bra*{\sup_{y \in \R^2} \int_{B_R(y)} \rho_k(x) \dif x } = 0.
\]
\item Dichotomy: there is $\lambda \in (0,M)$ such that, for every $\eta>0$, one can find radius $R>0$, centers $x_k \in \R^2$, and radii $S_k \to \infty$ such that 
\[
\abs*{\int_{B_R(x_k)} \rho_k \dif x - \lambda} < \eta, \qquad
\abs*{\int_{\R^2 \setminus B_{S_k}(x_k)} \rho_k \dif x - (M-\lambda)} < \eta
\]
and
\[
\int_{B_{S_k}(x_k) \setminus B_R(x_k)} \rho_k \dif x < \eta
\]
for all sufficiently large $k$.
\end{enumerate}
Vanishing can be excluded directly. Indeed, the local Sobolev estimate
\[
\norm{\n_k - \ein_3}_{L^\infty(B_1(y))} \le C \norm{\n_k - \ein_3}_{H^2(B_2(y))}, 
\]
where $C$ is independent of $k$ and $y$, shows that $\n_k\to \ein_3$ uniformly on $\R^2$. Hence, for sufficiently large $k$, the image of $\n_k$ is contained in the northern hemisphere, which implies $Q(\n_k)=0$, contradicting the assumption $Q(\n_k)=-1$.
This direct exclusion of vanishing is a benefit of the local $H^2$ control $H^2_{\mathrm{loc}}(\R^2)\hookrightarrow L^\infty_{\mathrm{loc}}(\R^2)$. It remains to exclude dichotomy.

\subsection{A sphere-valued \texorpdfstring{$H^2$}{H2} decomposition}\label{sec:decomposition}

The overall splitting strategy is adapted from the annular modification introduced by Lin and Yang for the two-dimensional Skyrme model \cite[Lemma~6.1]{LinYang2004}; see also \cite[Lemma~5.1 and Theorem~5.2]{LinYang2007}. In the present fourth-order setting, however, the cutoff must be
controlled in $H^2$, including through the sphere projection. We first select a fixed-width annulus of vanishing $H^2$-mass, then cut the inner and outer clusters off to $\ein_3$ through sphere-valued maps, and finally verify exact degree splitting and asymptotic energy splitting.

\begin{lemma}\label{la:small-annulus}
Suppose that the density sequence $(\rho_k)_{k \in \N}$ defined in \eqref{eq:rho-density} satisfies the dichotomy alternative. After translating the sequence and passing to a subsequence, not relabeled, there exist $r_k \to \infty$ and $\beta >0$ such that
\begin{equation}\label{eq:annulus-small}
\int_{B_{r_k+3} \setminus B_{r_k-2}} \rho_k \dif x \longrightarrow 0,
\end{equation}
while
\begin{equation}\label{eq:two-positive-masses}
\int_{B_{r_k-2}} \rho_k \dif x \ge \beta, \qquad 
\int_{\R^2 \setminus B_{r_k+3}} \rho_k \dif x \ge \beta.
\end{equation}
\end{lemma}

\begin{proof}
Let $\lambda \in (0,M)$ be the mass parameter in the dichotomy alternative and set
\[
\beta = \frac{1}{4} \min\set*{\lambda, M-\lambda}>0.
\]
Choose $\eta_j \to 0$ with $\eta_j \in (0,\beta]$. For each $j$, dichotomy provides an inner radius $R_j$, centers $x_{j,k}$, and outer radii $S_{j,k} \to \infty$ such that, for all sufficiently large $k$,
\[
 \int_{B_{S_{j,k}}(x_{j,k})\setminus B_{R_j}(x_{j,k})} \rho_k\dif x < \eta_j
\]
and
\[
 \int_{B_{R_j}(x_{j,k})}\rho_k\dif x > \lambda-\eta_j,\qquad 
 \int_{\R^2\setminus B_{S_{j,k}}(x_{j,k})}\rho_k\dif x > M-\lambda-\eta_j.
\]
Choose successively $k_j > k_{j-1}$ so large that the above inequalities hold and $S_{j,k_j} -R_j \ge 2j+6$. For each $j$, translate the diagonal subsequence by the corresponding centers $x_{j,k_j}$, without changing notation. Set
\[
S_j = S_{j,k_j}, \qquad r_j = \frac{1}{2}(R_j + S_j).
\]
Then $r_j \ge j+3$, and hence $r_j \to \infty$. The gap $S_j - R_j \ge 2j+6$ implies
\[
R_j \le r_j -2 < r_j + 3 \le S_j,
\]
for all sufficiently large $j$. Consequently,
\[
\int_{B_{r_j+3} \setminus B_{r_j-2}} \rho_{k_j} \dif x \le \int_{B_{S_j} \setminus B_{R_j}} \rho_{k_j} \dif x < \eta_j \longrightarrow 0,
\]
while the two complementary regions satisfy
\[
\int_{B_{r_j-2}} \rho_{k_j} \dif x \ge \int_{B_{R_j}} \rho_{k_j} \dif x > \lambda - \eta_j > \beta,
\]
and
\[
\int_{\R^2 \setminus B_{r_j+3}} \rho_{k_j} \dif x \ge \int_{\R^2 \setminus B_{S_j}} \rho_{k_j} \dif x > M-\lambda - \eta_j > \beta.
\]
\end{proof}

The separating annulus identifies a region where an annular cutoff can be used to separate the inner and outer parts of the map. Its extra width accommodates both the actual transition region and the buffered neighborhood needed for the $H^2$ estimates. The next lemma implements this cutoff and controls its $H^2$ cost on that buffered neighborhood.

\begin{lemma}\label{la:projection}
There exist constants $C, \eta_0>0$, independent of $r$ and $\n$, such that whenever $r \ge 3$ and  $\n \in \cA$ satisfy
\begin{equation}\label{eq:cutoff-smallness}
\eta:=\int_{B_{r+3} \setminus B_{r-2}} \bra*{|\n-\ein_3|^2 + |\nabla \n|^2 +|D^2 \n|^2} \dif x \le \eta_0,
\end{equation}
there exist $\ni, \no \in \cA$ such that
\[
\ni = \left\{ \begin{array}{l} \n \quad\text{on}\quad B_r, \\ \ein_3 \quad\text{on}\quad \R^2 \setminus B_{r+1}, \end{array} \right. \qquad
\no = \left\{ \begin{array}{l} \ein_3 \quad\text{on}\quad B_r, \\ \n \quad\text{on}\quad \R^2 \setminus B_{r+1}. \end{array} \right.
\]
Moreover, with $A_r := B_{r+2} \setminus B_{r-1}$,
\begin{equation}\label{eq:cutoff-H2-bound}
\norm{\ni - \ein_3}^2_{H^2(A_r)} + \norm{\no - \ein_3}^2_{H^2(A_r)} \le C\eta.
\end{equation}
\end{lemma}

\begin{proof}
Choose a smooth radial function $\chi_r: \R^2 \to [0,1]$ such that $\chi_r=1$ on $B_r$ and $\chi_r=0$ on $\R^2 \setminus B_{r+1}$, with first and second derivatives bounded uniformly in $r$. Denote $\bu=\n-\ein_3$ and set
\begin{equation}\label{eq:cutoff-map}
\ni = \frac{\ein_3 + \chi_r \bu}{|\ein_3 + \chi_r \bu|}, \qquad 
\no = \frac{\ein_3 + (1-\chi_r) \bu}{|\ein_3 + (1-\chi_r) \bu|}.
\end{equation}
We first derive the pointwise control needed to justify these definitions. 
Take $\zeta_r \in C_c^\infty(B_{r+3} \setminus B_{r-2})$ equal 1 on $A_r$, with first and second derivatives bounded uniformly in $r$. The product rule and \eqref{eq:cutoff-smallness} imply
\begin{equation}\label{eq:zeta-u}
\norm{\zeta_r \bu}_{H^2(\R^2)} \le C \eta^{1/2}
\end{equation}
By the embedding $H^2(\R^2) \hookrightarrow L^\infty(\R^2)$, 
\[
\norm{\bu}_{L^\infty(A_r)} \le \norm{\zeta_r \bu}_{L^\infty(\R^2)} \le C \norm{\zeta_r \bu}_{H^2(\R^2)} \le C \eta^{1/2}.
\]
Since $B_{r+1}\setminus B_r\subset A_r$, reducing $\eta_0$ ensures that the denominators in \eqref{eq:cutoff-map} are at least $1/2$ on the transition annulus $B_{r+1}\setminus B_r$. Outside this annulus, the vectors being normalized are either $\n$ or $\ein_3$, and hence have unit length. Thus \eqref{eq:cutoff-map} is well-defined on the whole plane. 

We next estimate the $H^2$ cost of projections on $A_r$. Let
\[
P(\bv) = \frac{\ein_3 + \bv}{|\ein_3 + \bv|} - \ein_3.
\]
On $|\bv| \le 1/2$, $P$ is smooth and
\[
|P(\bv)| \le C|\bv|, \quad |\nabla(P(\bv))| \le C|\nabla \bv|, \quad 
|D^2(P(\bv))| \le C(|D^2 \bv|+|\nabla \bv|^2).
\]
For $\bv = \chi_r \bu$ or $\bv = (1-\chi_r)\bu$, the product rule, the uniform derivative bounds
for $\chi_r$, $A_r \subset B_{r+3} \setminus B_{r-2}$ and \eqref{eq:cutoff-smallness} yield
\[
\norm{\bv}_{H^2(A_r)} \le C \eta^{1/2}.
\]
Using the same cutoff $\zeta_r$ and \eqref{eq:zeta-u}, we obtain
\[
\norm{\zeta_r \bv}_{H^2(\R^2)}  \le C \norm{\zeta_r \bu}_{H^2(\R^2)} \le C \eta^{1/2}.
\]
The Sobolev embedding $H^2(\R^2) \hookrightarrow W^{1,4}(\R^2)$ therefore implies
\[
\norm{|\nabla \bv|^2}_{L^2(A_r)} = \norm{\nabla \bv}_{L^4(A_r)}^2 \le C \norm{\zeta_r \bv}_{H^2(\R^2)}^2 \le C \eta.
\]
Combining these bounds and taking $\eta_0 \le 1$, we get
\[
\norm{P(\bv)}_{H^2(A_r)} \le C(\eta^{1/2}+\eta) \le C \eta^{1/2}
\]
and thus \eqref{eq:cutoff-H2-bound}. 

It remains only to verify that $\ni, \no \in \cA$. Since $\bu \in H^2$ and the cutoffs have bounded derivatives through order two, the product rule gives $\chi_r \bu, (1-\chi_r) \bu \in H^2$. The denominator bound, the embedding $H^2(\R^2)\hookrightarrow W^{1,4}(\R^2)$, and the chain rules for the normalization map then yield $\ni-\ein_3, \no-\ein_3 \in H^2(\R^2;\R^3)$. Since both maps are sphere-valued, they belong to $\cA$. This completes the proof.
\end{proof}

The preceding lemmas provide the separating radii and a $H^2$-estimate to isolate the inner and outer clusters. The following proposition records the resulting split maps and their degree, energy, and density properties. 

\begin{proposition}\label{prop:splitting}
Let $(\n_k)_{k \in \N}$ and its associated density sequence $(\rho_k)_{k \in \N}$ satisfy the assumptions of Lemma \ref{la:small-annulus}, and let $(r_k)_{k \in \N}$ and $\beta>0$ be provided by that lemma. Then, for all sufficiently large $k$, there exist maps $\ni_k, \no_k \in \cA$ satisfying
\begin{equation}\label{eq:agreement}
\ni_k = \left\{ \begin{array}{l} \n_k \quad\text{on}\quad B_{r_k}, \\ \ein_3 \quad\text{on}\quad \R^2 \setminus B_{r_k+1}, \end{array} \right. \qquad
\no_k = \left\{ \begin{array}{l} \ein_3 \quad\text{on}\quad B_{r_k}, \\ \n_k \quad\text{on}\quad \R^2 \setminus B_{r_k+1}. \end{array} \right.
\end{equation}
Their degrees split exactly:
\begin{equation}\label{eq:degree-split}
Q(\ni_k)+Q(\no_k)=Q(\n_k).
\end{equation}
Moreover,
\begin{equation}\label{eq:energy-split}
\sup_{0 \le \eps \le 2} \abs*{\cF_\eps(\ni_k) + \cF_\eps(\no_k) - \cF_\eps(\n_k) } \longrightarrow 0.
\end{equation}
If $\rho_k^{\mathrm{in}}$ and $\rho_k^{\mathrm{out}}$ denote the densities \eqref{eq:rho-density} associated with $\ni_k$ and $\no_k$, respectively, then
\begin{equation}\label{eq:cut-density}
\int_{B_{r_k-2}} \rho_k^{\mathrm{in}} \dif x \ge \beta, \qquad 
\int_{\R^2 \setminus B_{r_k+3}} \rho_k^{\mathrm{out}} \dif x \ge \beta.
\end{equation}
\end{proposition}

\begin{proof}
We denote
\[
T_k = B_{r_k+1} \setminus B_{r_k}, \qquad A_k = B_{r_k+2} \setminus B_{r_k-1}.
\]
Thus $T_k \subset A_k \subset B_{r_k+3} \setminus B_{r_k-2}$. These three nested annuli play different roles: $B_{r_k+3} \setminus B_{r_k-2}$ is the low-$H^2$-mass annulus supplied by Lemma \ref{la:small-annulus}, $T_k$ is the region where the cutoff varies and the maps are actually modified, and $A_k$ is a buffered neighborhood of $T_k$ on which the $H^2$ cost of the modification is estimated.

The smallness \eqref{eq:annulus-small} in Lemma \ref{la:small-annulus} permits application of Lemma \ref{la:projection} to $\n_k$ with $r=r_k$. It yields maps $\ni_k, \no_k \in \cA$ satisfying \eqref{eq:agreement} and
\begin{equation}\label{eq:cut-maps-small}
\norm{\ni_k - \ein_3}_{H^2(A_{r_k})}^2 + \norm{\no_k - \ein_3}_{H^2(A_{r_k})}^2 + \norm{\n_k - \ein_3}_{H^2(A_{r_k})}^2 \longrightarrow 0.
\end{equation}
The three degree integrals differ only on the transition annulus $T_k$. For every $\n \in \cA$,
\[
|\n \cdot (\partial_1 \n \times \partial_2 \n)| \le \frac{1}{2}|\nabla \n|^2.
\]
It follows from $T_k \subset A_k$ and \eqref{eq:cut-maps-small} that
\begin{align*}
 |Q(\ni_k)+Q(\no_k)-Q(\n_k)| 
\le  \frac{1}{8\pi} \int_{T_k} |\nabla \ni_k|^2 + |\nabla \no_k|^2 + |\nabla \n_k|^2 \dif x \to 0.
\end{align*}
The expression on the left is always an integer, so it vanishes for all sufficiently large $k$, proving \eqref{eq:degree-split}. 

Although $\cF_\eps$ is globally nonnegative for $0 \le \eps \le 2$, its local density is not pointwise nonnegative because of the Dirichlet energy term. For any measurable set $\Omega \subset \R^2$, let $\cF_\eps(\n;\Omega)$ denote the integral defining $\cF_\eps(\n)$, with the integration domain restricted to $\Omega$. We then use the following local estimate: 
\begin{align*}
&\abs*{\cF_\eps(\n; \Omega)} = \abs*{\frac{1}{2} \int_\Omega |\Delta \n|^2 + |\n-\ein_3|^2 - \eps|\nabla \n|^2 \dif x} \\
& \le \frac{1}{2} \int_\Omega |\Delta \n|^2 + |\n-\ein_3|^2 + \eps|\nabla \n|^2 \dif x
 \le C \norm{\n-\ein_3}_{H^2(\Omega)}^2
\end{align*}
for a constant $C>0$ independent of $\eps$.
By the exact agreement \eqref{eq:agreement}, $\cF_\eps(\ni_k) + \cF_\eps(\no_k) - \cF_\eps(\n_k)$ is supported in $T_k$. The preceding local estimate, the inclusion $T_k \subset A_k$ and \eqref{eq:cut-maps-small} imply that
\begin{align*}
&\abs*{\cF_\eps(\ni_k) + \cF_\eps(\no_k) - \cF_\eps(\n_k)}\\
& \le C \bra*{\norm{\ni_k - \ein_3}_{H^2(A_{r_k})}^2 + \norm{\no_k - \ein_3}_{H^2(A_{r_k})}^2 + \norm{\n_k - \ein_3}_{H^2(A_{r_k})}^2}\\
& \longrightarrow 0.
\end{align*}
This proves \eqref{eq:energy-split}. Finally, $B_{r_k-2} \subset B_{r_k}$ and $\R^2 \setminus B_{r_k+3} \subset \R^2 \setminus B_{r_k+1}$. The agreement properties \eqref{eq:agreement} therefore give
\[
\int_{B_{r_k-2}} \rho^{\mathrm{in}}_k \dif x = \int_{B_{r_k-2}} \rho_k \dif x \ge \beta, \quad
\int_{\R^2 \setminus B_{r_k+3} } \rho^{\mathrm{out}}_k \dif x = \int_{\R^2 \setminus B_{r_k+3} } \rho_k \dif x \ge \beta,
\]
where the inequalities follow from \eqref{eq:two-positive-masses} in Lemma \ref{la:small-annulus}. This is \eqref{eq:cut-density}.
\end{proof}

\subsection{A threshold-stable binding inequality}\label{sec:binding}
 
By Proposition \ref{prop:splitting}, it remains to exclude dichotomy into two nonzero-degree components. We obtain the required binding gap by combining a uniform lower bound for the Helmholtz residual with an explicit degree-$-1$ comparison map.

\subsubsection{Helmholtz circle means and the south-pole cost}

Define the Helmholtz residual by
\[
\cR(\n):=(\Delta+1)(\n-\ein_3).
\]
Thus $\cR(\n)$ measures the failure of $\n-\ein_3$ to satisfy the homogeneous Helmholtz equation. The homogeneous part of the circle-mean identity is the classical Helmholtz mean-value formula \cite{Kuznetsov2021}; the inhomogeneous representation and its $L^2$-estimate are derived here.
\begin{lemma}\label{la:circle-mean}
For $\n \in \cA$, $x_0 \in \R^2$ and $r>0$, we denote the mean value of $\bs n$ over the circle $\partial B_r(x_0)$ by
\[
\widetilde{\bs n}(x_0,r) = \frac{1}{2\pi} \int_0^{2\pi} \n(x_0 + r(\cos \varphi, \sin \varphi)) \dif \varphi.
\]
Then
\begin{equation}\label{eq:circle-mean}
\widetilde{\bs n}(x_0,r) = J_0(r) \n(x_0) + (1-J_0(r)) \ein_3 + \widetilde{\bs e}(x_0,r),
\end{equation}
where the error term $\widetilde{\bs e}(x_0,r)$ is generated by $\cR(\n)$,
\begin{equation}\label{eq:circle-error}
\widetilde{\bs e}(x_0,r)=\frac{1}{2\pi} \int_{B_r(x_0)} G_r(|y-x_0|) \cR(\n)(y) \dif y,
\end{equation}
and $G_r$ is the associated Green kernel given by
\[
G_r(t)=\frac{\pi}{2}\pra*{J_0(t)Y_0(r) - Y_0(t)J_0(r) },
\]
where $J_0$ and $Y_0$ are the Bessel functions of order zero of the first and second kind, respectively.
Moreover,
\begin{equation}\label{eq:kernel-norm}
\abs*{\widetilde{\bs e}(x_0,r)} \le \bra*{\frac{1-J_0(r)^2}{4\pi}}^{1/2} \norm{\cR(\n)}_{L^2(B_r(x_0))}.
\end{equation}
\end{lemma}

\begin{proof}
First we assume that $\bu=\n-\ein_3$ is smooth and take $x_0=0$ by translation. Define
\[
\widetilde{\bs u}(r)=\frac{1}{2\pi} \int_0^{2\pi} \bu(r \cos \varphi, r \sin \varphi) \dif \varphi
\]
and
\[
\widetilde{\cR}(r)=\frac{1}{2\pi} \int_0^{2\pi} \cR(\n)(r \cos \varphi, r \sin \varphi) \dif \varphi.
\]
Averaging the equation $(\Delta+1)\bu=\cR(\n)$ over $\partial B_r(x_0)$ gives
\[
\widetilde{\bs u}''(r) + \frac{1}{r}\widetilde{\bs u} '(r) + \widetilde{\bs u}(r)=\widetilde{\cR}(r), \quad \widetilde{\bs u}(0)=\n(0)-\ein_3, \quad \widetilde{\bs u}'(0)=0.
\]
The fundamental solutions of the homogeneous equation are $J_0$ and $Y_0$, and their Wronskian is (see \cite[(10.5.2)]{NISTDLMF})
\[
J_0(r) Y'_0(r)-J'_0(r)Y_0(r)=\frac{2}{\pi r}.
\]
Regularity of $\widetilde{\bs u}$ at $r=0$ excludes a homogeneous $Y_0$-term. Therefore, variation of parameters gives
\begin{align*}
\widetilde{\bs u}(r) &= J_0(r) \bu(0) + \frac{\pi}{2} \int_0^r \pra*{J_0(t)Y_0(r) - Y_0(t)J_0(r)} \widetilde{\cR}(t)\, t \dif t\\
&=J_0(r) \bu(0) + \int_0^r G_r(t) \widetilde{\cR}(t)\, t \dif t\\
&=J_0(r) \bu(0) + \frac{1}{2\pi} \int_{B_r(0)}G_r(|y|) \cR(\n)(y) \dif y.
\end{align*}
Since $\widetilde{\bs n}(0,r)=\ein_3 + \widetilde{\bs u}(r)$ and $\bu(0)=\n(0)-\ein_3$, this proves \eqref{eq:circle-mean} and \eqref{eq:circle-error} for smooth $\n$.

Next we compute the norm of the kernel. Its definition and the Bessel Wronskian imply
\begin{equation}\label{eq:gr1}
G_r''(t) + \frac{1}{t} G_r'(t)  + G_r(t) = 0, \qquad G_r(r)=0, \quad G'_r(r)=-\frac{1}{r}.
\end{equation}
The standard expansion of $J_0$ and $Y_0$ at the origin gives, for any fixed $r>0$,
\begin{equation}\label{eq:gr2}
t G_r(t) \to 0, \quad t G'_r(t) \to -J_0(r) \quad\text{as } t \to 0.
\end{equation}
Using \eqref{eq:gr1} for $G_r$  gives
\[
\frac{\dif}{\dif t} \pra*{\frac{t^2}{2} (G_r(t)^2 + G'_r(t)^2) } = t G_r(t)^2.
\]
Combining with \eqref{eq:gr2} yields
\begin{align*}
\int_{B_r} G_r(|y|)^2 \dif y &= 2\pi\int_0^r G_r(t)^2 t \dif t \\
&= 2\pi\bra*{\frac{r^2}{2}\pra*{G_r(r)^2+G'_r(r)^2} - \frac{1}{2} \lim_{t \to 0} \pra*{(t G_r(t))^2 + (tG'_r(t))^2}}\\
& = \pi \bra*{1-J_0(r)^2}.
\end{align*}
The Cauchy--Schwarz inequality therefore gives
\begin{align*}
|\widetilde{\bs e}(0,r)| & \le \frac{1}{2\pi} \bra*{\int_{B_r} G_r(|y|)^2 \dif y}^{1/2} \bra*{\int_{B_r} |\cR(\n)(y)|^2 \dif y}^{1/2} \\
& \le \frac{1}{2\pi} \bra*{\pi \bra*{1-J_0(r)^2}}^{1/2}\norm{\cR(\n)}_{L^2(B_r)},
\end{align*}
which is \eqref{eq:kernel-norm}.

The general case follows by mollification, since the preceding derivation uses only the linear Helmholtz equation and does not require the sphere constraint. Take $\bu_\delta = \bu * \eta_\delta \to \bu$ in $H^2 $, while $(\Delta+1) \bu_\delta \to (\Delta+1) \bu$ in $L^2$. Since $G_r \in L^2(B_r)$, all terms in the circle-mean identity pass to the limit.
\end{proof}

The next lemma converts the circle-mean estimate into a lower bound for the energy $\cF_\eps$. Recall that $r_*$ is the first positive zero of $J_1$, and set $\alpha=-J_0(r_*)$.  By the interlacing properties of the positive zeros of $J_0$ and $J_1$, we have $J_0(r_*)<0$, and hence $\alpha\in(0,1)$.
\begin{lemma}\label{la:south-cost}
 If $\n \in \cA$ and $Q(\n) \neq 0$, then
 \begin{equation}\label{eq:south-cost}
 \frac{1}{2}\norm{\cR(\n)}^2_{L^2(\R^2)} \ge  b_* = \frac{8\pi \alpha^2}{1-\alpha^2}.
 \end{equation}
 Moreover, for $0 \le \eps \le 2$,
 \begin{equation}\label{eq:lower-bound}
 \cF_\eps(\n) \ge  b_* + 8\pi \bra*{1-\frac{\eps}{2}}|Q(\n)| \quad\text{if } Q(\n) \neq 0.
 \end{equation}
\end{lemma}

\begin{proof}
A continuous map $f: \St \to \St$ of nonzero degree is surjective. Indeed, if $\bs p \notin f(\St)$, then $f$ factors through the contractible set $\St \setminus \set{\bs p}$. Hence $f$ is null-homotopic and therefore has degree zero. Applying this to $\n$ gives $\n(x_0)=-\ein_3$ for some $x_0 \in \R^2$. Taking $r=r_*$ in Lemma \ref{la:circle-mean} and using $J_0(r_*)=-\alpha$, we obtain
\[
\widetilde{\bs n}(x_0,r_*) = (1+2\alpha) \ein_3 + \widetilde{\bs e}(x_0,r_*).
\]
Since $|\n|=1$, the triangle inequality gives
\[
|\widetilde{\bs n}(x_0,r_*) | \le \frac{1}{2\pi}\int_0^{2\pi}|\n|\dif \varphi =1.
\]
Hence, by the reverse triangle inequality,
\[
1\ge |(1+2\alpha)\ein_3+\widetilde{\bs e}(x_0,r_*)|
\ge 1+2\alpha-|\widetilde{\bs e}(x_0,r_*)|,
\]
and therefore $|\widetilde{\bs e}(x_0,r_*)|\ge2\alpha$. Combining this with \eqref{eq:kernel-norm} yields
\[
\norm{\cR(\n)}_{L^2(\R^2)}^2 \ge \norm{\cR(\n)}_{L^2(B_{r_*}(x_0))}^2 \ge \frac{16\pi \alpha^2}{1-\alpha^2}.
\]
This gives \eqref{eq:south-cost}. The residual representation \eqref{eq:Feps-residual} and the classical topological lower bound $\norm{\nabla \n}_{L^2(\R^2)}^2 \ge 8\pi |Q(\n)| $ give \eqref{eq:lower-bound}.
\end{proof}

The south-pole value $\n(x_0)=-\ein_3$ and radius $r_*$ are chosen to optimize the lower bound: the antipodal value $-\ein_3$ maximizes the distance of the homogeneous mean from the unit ball, while $r=r_*$ maximizes $|J_0(r)|$ among radii for which $J_0(r)<0$. Other such choices would give valid but weaker residual bounds.

\subsubsection{An upper bound for \texorpdfstring{$\widehat{I}_{-1}(\eps)$}{I-1(eps)}}

We next derive an upper bound for $\widehat{I}_{-1}(\eps)$ by means of an explicit radial test map. Combined with the lower bound in Lemma \ref{la:south-cost}, this gives the strict binding inequality required for compactness.

Define
\[
\bw_\rho(r,\varphi) = \begin{pmatrix}
\sin f(r/\rho) \cos\varphi \\
\sin f(r/\rho) \sin\varphi\\
\cos f(r/\rho)
\end{pmatrix}
\quad\text{where}\quad
f(t) = \left\{ 
\begin{array}{ll}
 \pi(1-t)^2, & 0 \le t \le 1,\\
 0, & t \ge 1.
\end{array}
 \right.
\]
Near the origin, $\sin f(t)=2\pi t -\pi t^2+O(t^3)$. 
At $r=\rho$, the map and the first derivative match the constant background $\ein_3$. Consequently, $\bw_\rho-\ein_3\in H^2(\R^2;\R^3)$ is compactly supported, while the radial degree formula gives $Q(\bw_\rho)=-1$. 

Define
\[
e(t)=f'(t)^2 + \frac{\sin^2 f(t)}{t^2}, \qquad \tau(t)=f''(t) + \frac{f'(t)}{t}-\frac{\sin f(t) \cos f(t)}{t^2}.
\]
Define the following energy quantities associated with $\bw_1$:
\begin{align*}
D_0&=\frac{1}{2}\int_{\R^2}|\nabla \bw_1|^2\dif x =\pi\int_0^1e(t)t\dif t, \\
 A_0&=\int_{\R^2}|\Delta \bw_1|^2\dif x 
 =2\pi\int_0^1\bigl(\tau(t)^2+e(t)^2\bigr)t\dif t, \\
 P_0&=\int_{\R^2}|\bw_1-\ein_3|^2\dif x =4\pi\int_0^1(1-\cos f(t))t\dif t.
\end{align*}
It follows from scaling that, for any $\rho>0$, 
\[
\cF_\eps(\bw_\rho) = \frac{1}{2\rho^2} A_0 - \eps D_0 + \frac{\rho^2}{2} P_0.
\]
The positive terms are minimized at
\begin{equation}\label{eq:trial-energy}
\rho_* = \bra*{\frac{A_0}{P_0}}^{1/4}, \quad\text{with}\quad
\cF_\eps(\bw_{\rho_*}) = \sqrt{A_0 P_0} - \eps D_0.
\end{equation}
Elementary one-dimensional estimates for the preceding integrals, together with bounds on the first negative minimum of $J_0$, give
\[
14.85<D_0<12\pi, \quad A_0 < 968, \quad P_0<1.59, \quad 0.402<\alpha<0.403.
\]
These estimates imply
\begin{equation}\label{eq:endpoint-gap}
\sqrt{A_0 P_0} - 2D_0 < 9.54 < 9.68 < 2 b_*.
\end{equation}
We define
\[
\delta_*:= 2 b_*-\bra*{\sqrt{A_0P_0}-2D_0}>0.
\]
Thus the radial test map lies strictly below the endpoint lower bound for any nontrivial degree splitting. The following proposition shows that this gap persists uniformly for $0 \le \eps \le 2$.

\begin{proposition}
For every $0 \le \eps \le 2$,
\begin{equation}\label{eq:strict-binding}
\widehat{I}_{-1}(\eps) \le \cF_\eps(\bw_{\rho_*}) \le 2 b_*+24\pi\left(1-\frac\eps2\right)-\delta_*.
\end{equation}
\end{proposition}

\begin{proof}
Since $\bw_{\rho_*} \in \cA_{-1}$, \eqref{eq:trial-energy} gives
\[
\widehat{I}_{-1} (\eps)\le \cF_\eps(\bw_{\rho_*}) = \sqrt{A_0 P_0} -\eps D_0.
\]
In a nontrivial topological splitting of degree $-1$, the two components have degrees $q$ and $-1-q$ for some $q \in \Z \setminus \set*{0,-1}$. Applying \eqref{eq:lower-bound} to both components and using $|q|+|-1-q| \ge 3$ gives the lower bound
\[
2 b_* + 8\pi \bra*{1-\frac{\eps}{2}} \bra*{|q|+|-1-q|} \ge 2 b_* + 24\pi \bra*{1-\frac{\eps}{2}} .
\]
The difference between the splitting barrier and the energy of the trial map is
\[
2 b_* + 24\pi \bra*{1-\frac{\eps}{2}} - \cF_\eps(\bw_{\rho_*}) 
\ge \delta_*+(2-\eps)(12\pi-D_0) > 0
\]
since $D_0<12\pi$. This proves \eqref{eq:strict-binding}.
\end{proof}

\subsection{Compactness and attainment}\label{sec:existence}

We now combine the preceding results into a compactness theorem which is useful both at fixed field and in the strong-field limit. 

\begin{theorem}\label{thm:uniform-compactness}
Let $K\Subset[0,2)$, and suppose that
$\eps_k\in K$, $\n_k\in\cA_{-1}$, and
\begin{equation}\label{eq:almost-minimizers}
\cF_{\eps_k}(\n_k)-\widehat I_{-1}(\eps_k)\longrightarrow0.
\end{equation}
Then every subsequence has a further subsequence along which there exist a parameter $\eps_*\in K$, translations $a_k\in\R^2$, and a map
$\n_*\in\cA_{-1}$ such that
\[
\eps_k\longrightarrow\eps_*,
\qquad
\n_k(\,\cdot+a_k)-\ein_3
\longrightarrow
\n_*-\ein_3
\quad\text{strongly in }H^2(\R^2;\R^3).
\]
Every limit obtained in this way satisfies
\begin{equation}\label{eq:limit-minima}
Q(\n_*)=-1,
\qquad
\cF_{\eps_*}(\n_*)=\widehat I_{-1}(\eps_*),
\end{equation}
and, along the same subsequence,
\[
\widehat I_{-1}(\eps_k)
\longrightarrow
\widehat I_{-1}(\eps_*).
\]
\end{theorem}

\begin{proof}
Fix an arbitrary subsequence. By \eqref{eq:F0-comparison}, \eqref{eq:strict-binding}, and \eqref{eq:almost-minimizers}, the sequence $(\n_k-\ein_3)$ is uniformly bounded in $H^2(\R^2;\R^3)$, and hence the total masses $\int_{\R^2}\rho_k\,\dif x$ are uniformly bounded. Since $K$ is compact, after passing to a further subsequence, not relabeled, we may assume that
\[
\eps_k\longrightarrow\eps_*\in K,
\qquad
\int_{\R^2}\rho_k\,\dif x\longrightarrow M
\]
for some $M\ge0$.
The topological lower bound gives $M\ge 8\pi$, while the local $H^2$ argument following the concentration-compactness alternatives excludes vanishing.

Suppose that dichotomy occurs. By Proposition \ref{prop:splitting}, we obtain split maps $\ni_k$ and $\no_k$.  By \eqref{eq:energy-split}, uniform coercivity on $K$, and the classical topological lower bound, their integer degrees are uniformly bounded. Hence, after passing to a further subsequence, we may assume that
\[
 Q(\ni_k)=q,\qquad Q(\no_k)=-1-q,
\]
for some $q \in \Z$.
If $q \notin \set{0,-1}$, both pieces have nonzero degree and $|q|+|-1-q| \ge 3$. Hence the energy splitting \eqref{eq:energy-split} and the energy barrier \eqref{eq:lower-bound} give
\begin{align*}
\cF_{\eps_k}(\n_k) & \ge 2 b_* + 8\pi \bra*{1-\frac{\eps_k}{2}}(|q|+|-1-q|)+o(1)\\
& \ge 2 b_* + 24\pi\bra*{1-\frac{\eps_k}{2}} +o(1).
\end{align*}
By \eqref{eq:strict-binding},
\[
2 b_*+24\pi\left(1-\frac{\eps_k}{2}\right)
\ge \widehat I_{-1}(\eps_k)+\delta_*.
\]
Consequently,
\[
\cF_{\eps_k}(\n_k)
\ge \widehat I_{-1}(\eps_k)+\delta_*+o(1),
\]
which contradicts \eqref{eq:almost-minimizers}.

If $q=0$, then $Q(\no_k)=-1$, and hence $\cF_{\eps_k}(\no_k) \ge \widehat I_{-1}(\eps_k)$.
On the other hand, Proposition \ref{prop:splitting} gives $ \|\ni_k - \ein_3\|_{H^2}^2\ge\beta$. Since $K\Subset[0,2)$, the coercivity estimate \eqref{eq:F0-comparison} is uniform on $K$, and consequently,
\[
 \cF_{\eps_k}(\ni_k) \ge c_K\|\ni_k-\ein_3\|_{H^2}^2 \ge c_K\beta>0.
\]
Energy splitting \eqref{eq:energy-split} therefore yields
\[
 \cF_{\eps_k}(\n_k) \ge\widehat I_{-1}(\eps_k)+c_K\beta+o(1),
\]
contradicting \eqref{eq:almost-minimizers}. The case $q = -1$ is identical after interchanging the pieces. Thus dichotomy is also impossible. 

The compactness alternative therefore holds. Choose translations $a_k\in\R^2$, set
\[
\tilde{\n}_k:=\n_k(\cdot+a_k), \quad \bu_k:=\tilde{\n}_k-\ein_3,
\]
and note that the densities associated with $\tilde{\n}_k$ are tight. After passing to a further subsequence, there exists $\bu_* \in H^2(\R^2;\R^3)$ such that
\begin{equation}\label{eq:weak-H2}
\bu_k \rightharpoonup \bu_* \quad\text{weakly in } H^2(\R^2;\R^3).
\end{equation}
Set $\n_*:=\bu_*+\ein_3$. The Rellich-Kondrachov theorem gives
\[
\bu_k \longrightarrow \bu_* \quad\text{strongly in } H^1(B_R) 
\]
for every fixed $R>0$. Tightness makes the $H^1$-tails of $\bu_k$ uniformly small, while
weak lower semicontinuity gives the corresponding tail bound for $\bu_*$. Combining this with the local convergence yields
\[
\bu_k \longrightarrow \bu_* \quad\text{strongly in } H^1(\R^2;\R^3).
\]
Since $\tilde{\n}_k \to \n_*$ strongly in $L^2$ and $|\tilde{\n}_k|=1$ a.e., we have $|\n_*|=1$ a.e. The degree is continuous under strong $H^1$-convergence, while translations preserve degree. Thus $Q(\n_*)=\lim_{k \to \infty} Q(\tilde{\n}_k) =-1$, so that $\n_*\in\cA_{-1}$.

Finally, we identify the limiting energy in order to upgrade the weak convergence in $H^2$ to strong convergence. The strong $H^1$-convergence gives $D(\tilde{\n}_k) \to D(\n_*)$. Consequently,
\[
\cF_{\eps_*}(\tilde{\n}_k) - \cF_{\eps_k}(\tilde{\n}_k) = (\eps_k - \eps_*) D(\tilde{\n}_k) \to 0.
\]
By translation invariance and \eqref{eq:almost-minimizers}, it follows that
\[
\cF_{\eps_*}(\tilde{\n}_k) = \widehat I_{-1}(\eps_k)+o(1).
\]
Weak lower semicontinuity therefore gives
\[
\widehat I_{-1}(\eps_*) \le \cF_{\eps_*}(\n_*)
\le \liminf_{k\to\infty} \cF_{\eps_*}(\tilde{\n}_k)
= \liminf_{k\to\infty}\widehat I_{-1}(\eps_k).
\]
Conversely, for every fixed $\bw\in\cA_{-1}$,
\[
\limsup_{k\to\infty}\widehat I_{-1}(\eps_k)
\le \lim_{k\to\infty}\cF_{\eps_k}(\bw) 
= \cF_{\eps_*}(\bw).
\]
Taking the infimum over $\bw\in\cA_{-1}$, we obtain
\[
\limsup_{k\to\infty}\widehat I_{-1}(\eps_k)
\le \widehat I_{-1}(\eps_*).
\]
Thus \eqref{eq:limit-minima} holds and $\cF_{\eps_*}(\tilde{\n}_k) \to \cF_{\eps_*}(\n_*) = \widehat I_{-1}(\eps_*)$.

It remains to upgrade the convergence. Define the symmetric bilinear form
\[
B_{\eps_*}(\bv,\bw)
:= \int_{\R^2} \Delta\bv\cdot\Delta\bw + \bv\cdot\bw - \eps_*\nabla\bv\cdot\nabla\bw \,\dif x.
\]
Since $\eps_*\in K\Subset[0,2)$, the interpolation estimate \eqref{eq:interpolation} implies that
\begin{equation}\label{eq:B_coercive}
B_{\eps_*}(\bv,\bv) \ge c_K\|\bv\|_{H^2(\R^2)}^2
\end{equation}
for every $\bv\in H^2(\R^2;\R^3)$, with $c_K>0$. Thus $B_{\eps_*}$ defines a norm equivalent to the standard $H^2$-norm. The weak convergence \eqref{eq:weak-H2} gives
\[
B_{\eps_*}(\bu_k,\bu_*) \longrightarrow B_{\eps_*}(\bu_*,\bu_*),
\]
while convergence of the energies gives
\[
B_{\eps_*}(\bu_k,\bu_k)
= 2\cF_{\eps_*}(\tilde{\n}_k) \longrightarrow
2\cF_{\eps_*}(\n_*) = B_{\eps_*}(\bu_*,\bu_*).
\]
Hence
\begin{align*}
B_{\eps_*}(\bu_k-\bu_*,\bu_k-\bu_*)
= B_{\eps_*}(\bu_k,\bu_k)-2B_{\eps_*}(\bu_k,\bu_*)+B_{\eps_*}(\bu_*,\bu_*)
\longrightarrow 0.
\end{align*}
By coercivity \eqref{eq:B_coercive}, $\bu_k\to\bu_*$ strongly in $H^2(\R^2;\R^3)$. This completes the proof.
\end{proof}

\begin{proof}[Proof of Theorem~\ref{thm:existence}]
For fixed $0\le\eps<2$, apply Theorem~\ref{thm:uniform-compactness} to a minimizing sequence for $\widehat I_{-1}(\eps)$ with $\eps_k\equiv\eps$. This gives attainment of $\widehat I_{-1}(\eps)$ and strong $H^2$ compactness of minimizing sequences modulo translations.

For $0<\eps<2$, set $H=\eps^{-2}$. The assertions for $E_H$ follow from the degree-preserving dilation $\m(x)=\n(H^{1/4}x)$ and the exact identities \eqref{eq:exact}.
\end{proof}

\section{The coercivity threshold}\label{sec:threshold}

This section describes the variational problem at and around the threshold $H=1/4$. We first quantify the degeneration of $H^2$-coercivity as $H\downarrow1/4$. We then construct degree-zero Weyl sequences that exhibit the loss of endpoint tightness and use the same critical oscillatory mode to prove that the energy is unbounded below for $0<H<1/4$. Finally, we establish a positive endpoint energy barrier in every nonzero degree class and prove right-continuity of $I_{-1}(H)$ at the threshold.

At the endpoint, completing the square in the original variables gives
\begin{equation}\label{eq:endpoint-residual}
E_{1/4}(\m) = \frac{1}{2} \norm*{\bra*{\Delta+\frac{1}{2}} (\m - \ein_3) }_{L^2(\R^2)}^2.
\end{equation}

\subsection{Coercivity and endpoint Weyl sequences}

The precise coercivity estimate is as follows.
\begin{proposition}\label{prop:best-coercivity}
For every $H>1/4$ and every $\m\in\cA$,
\[
E_H(\m) \ge c_H \norm{\m-\ein_3}_{H^2(\R^2)}^2,
\]
where
\begin{equation}\label{eq:best-coercivity}
 c_H=\frac{1}{2}+\frac{H}{3}-\frac{1}{3}\sqrt{H^2+3}>0.
\end{equation}
Moreover, as $H\downarrow1/4$,
\begin{equation}\label{eq:coercivity-asymptotic}
 c_H=\frac{2}{7}\bra*{H-\frac{1}{4}} +O\bra*{\bra*{H-\frac{1}{4}}^2}.
\end{equation}
\end{proposition}

\begin{proof}
Let $\m\in\cA$ and write $\bu=\m-\ein_3$. Since $|\m|=1$, $|\bu|^2=2(1-m_3)$. Plancherel's theorem therefore gives
\begin{align*}
 2E_H(\m)
=\int_{\R^2}\bra*{ |\Delta \bu|^2-|\nabla \bu|^2+H|\bu|^2}\dif x
=\int_{\R^2} \bra*{ |\xi|^4-|\xi|^2+H} |\widehat{\bu}(\xi)|^2\dif \xi.
\end{align*}
In this proof, we use the equivalent $H^2$ norm
\[
 \norm{\bv}_{H^2}^2=\int_{\R^2}  (1+|\xi|^2+|\xi|^4)|\widehat \bv(\xi)|^2\dif\xi.
\]
Writing $s=|\xi|^2$, it follows that
\[
 2E_H(\m)\ge \bra*{\inf_{s \ge 0} \frac{s^2-s+H}{1+s+s^2}} \norm{\m-\ein_3}_{H^2(\R^2)}^2.
\]
The quotient attains its minimum at $(H-1+\sqrt{H^2+3})/2$. Substitution gives \eqref{eq:best-coercivity}, and Taylor expansion at $H=1/4$ gives \eqref{eq:coercivity-asymptotic}.
\end{proof}

Proposition~\ref{prop:best-coercivity} shows that the $H^2$-coercivity constant collapses as $H\downarrow1/4$. The following construction makes this loss of endpoint coercivity explicit.

\begin{proposition}\label{prop:weyl}
There exist $R_0,c_0,C>0$ and a family $(\bw_R)_{R \ge R_0} \subset \cA_0$ such that
\begin{equation}\label{eq:weyl-energy}
\|\bw_R-\ein_3\|_{H^2(\R^2)}\ge c_0, \qquad
0\le E_{1/4}(\bw_R)\le CR^{-2}.
\end{equation}
Moreover, for every $A>0$, there exists a constant $C_A>0$, independent of $R$, such that
\begin{equation}\label{eq:weyl-vanishing}
\sup_{y \in \R^2} \int_{B_A(y)} \bra*{|\bw_R - \ein_3|^2 + |\nabla \bw_R|^2 +|D^2 \bw_R|^2} \dif x \le C_A R^{-2}.
\end{equation}
At fields of the form $H=1/4+\delta$, the same family satisfies
\begin{equation}\label{eq:weyl-delta}
 0 \le E_{1/4+\delta}(\bw_R) \le C\bra*{R^{-2}+\delta},
\end{equation}
for every $\delta \ge 0$ and every $R \ge R_0$.
\end{proposition}

\begin{proof}
Fix a nonzero function $\chi \in C_c^\infty(\R^2)$ and set $k=1/\sqrt{2}$. Define
\[
a_R(x)=\frac{1}{R} \chi\bra*{\frac{x}{R}}, \qquad v_R(x)=a_R(x)(\cos(k x_1), \sin(k x_1))
\]
and
\[
b_R(x) = \sqrt{1-a_R^2(x) }-1, \qquad \bw_R(x) = \ein_3 + \bra*{v_R(x), b_R(x)}.
\]
For large $R>0$, $\bw_R \in \cA_0$, since its image lies in the northern hemisphere. A change of variables gives $\norm{v_R}_{L^2(\R^2)} = \norm{\chi}_{L^2(\R^2)}>0$. Consequently,
\[
\|\bw_R-\ein_3\|_{H^2(\R^2)} \ge \norm{\chi}_{L^2(\R^2)} =: c_0>0.
\]
For $\ell=0,1,2$,
\[
|D^{\ell} v_R(x)| \le C R^{-1}, \qquad |D^{\ell} b_R(x)| \le C R^{-2}.
\]
Therefore, integration over any fixed-radius ball gives \eqref{eq:weyl-vanishing}.
Direct differentiation gives
\[
\norm{(\Delta+k^2)v_R}_{L^2(\R^2)} \le 2k \norm{\partial_1 a_R}_{L^2(\R^2)}  + \norm{\Delta a_R}_{L^2(\R^2)} \le C R^{-1},
\]
and the chain rule gives
\[
\norm{(\Delta+k^2)b_R}_{L^2(\R^2)}\le C R^{-1}.
\]
Since $k^2=1/2$, the endpoint identity \eqref{eq:endpoint-residual} gives
\[
0 \le 2E_{1/4}(\bw_R) =\norm*{\bra*{\Delta+\frac{1}{2}}v_R}_{L^2(\R^2)}^2 + \norm*{\bra*{\Delta+\frac{1}{2}}b_R}_{L^2(\R^2)}^2 \le C R^{-2}.
\]
This completes the proof of \eqref{eq:weyl-energy}. Finally,
\[
E_{1/4+\delta}(\bw_R) = E_{1/4}(\bw_R) + \delta \int_{\R^2} 1- w_{R,3} \dif x,
\]
and
\[
\int_{\R^2} 1- w_{R,3} \dif x = \int_{\R^2} 1-\sqrt{1-a_R^2} \dif x \le \int_{\R^2} a_R^2 \dif x = \norm{\chi}^2_{L^2(\R^2)}.
\]
After enlarging $C$ if necessary, this proves \eqref{eq:weyl-delta}.
\end{proof}

For all sufficiently small $\delta>0$, taking $R=\delta^{-1/2}$ in \eqref{eq:weyl-delta} gives $ E_{1/4+\delta}(\bw_R) \le 2C\delta \to 0$ as $\delta \downarrow 0$. At the same time, \eqref{eq:weyl-energy} and \eqref{eq:weyl-vanishing} show that the total $H^2$ mass remains bounded away from zero, while its mass in every fixed-radius ball tends to zero uniformly with respect to its center. Thus the coercivity estimate used to exclude degree-zero remainders in Section~\ref{sec:existence} is not uniform as $H\downarrow1/4$.

\subsection{Instability below the threshold}

The Weyl family constructed in Proposition~\ref{prop:weyl} is the vanishing-amplitude endpoint analogue of the conical-spiral background below the threshold. Both are built from the helical mode $k=1/\sqrt{2}$, whose quadratic energy vanishes at $H=1/4$ and becomes negative for $H<1/4$. By filling an expanding region with a fixed small amplitude of this mode, we obtain energies tending to $-\infty$.

\begin{proposition}\label{prop:below-threshold}
If $0<H<1/4$, then $I_q(H)=-\infty$ for every $q\in\Z$.
\end{proposition}

\begin{proof}
Fix $H<1/4$. Choose $\eta_R \in C_c^\infty(B_{R+1})$ satisfying
\[
\eta_R(x)\equiv 1 \text{ on } B_R, \quad 0 \le \eta_R(x) \le 1, \quad\text{and }
\|D\eta_R\|_{L^\infty}+\|D^2\eta_R\|_{L^\infty}\le C
\]
with $C$ independent of $R$. For $k=1/\sqrt{2}$ and $a \in (0,1/2)$, define
\[
\m_R = \bra*{a \eta_R(x) \cos(k x_1), \, a \eta_R(x) \sin(k x_1), \, \sqrt{1-a^2 \eta_R(x)^2} }.
\]
Its image lies in the northern hemisphere, so $Q(\m_R)=0$. On $B_R$, the energy density of $\m_R$ is
\[
e_{H,a} = -\frac{a^2}{8} + H(1-\sqrt{1-a^2}).
\]
Since 
\[
\frac{e_{H,a}}{a^2} = -\frac{1}{8}+\frac{H}{1+\sqrt{1-a^2}} \longrightarrow -\frac{1}{8}+\frac{H}{2} = \frac{1}{2}\bra*{H-\frac{1}{4}} \quad\text{as } a \to 0,
\]
we can choose $a$ so small that $e_{H,a} <0$. The transition annulus has area $O(R)$ and uniformly bounded derivatives of $\m_R$. Thus
\[
E_H(\m_R) \le \pi R^2 e_{H,a} + CR \longrightarrow -\infty.
\]
For arbitrary $q \in \Z$, choose a smooth map in $\cA_q$ that equals $\ein_3$ outside a compact set.  After translating this map so that its nonconstant region is disjoint from that of $\m_R$, glue both maps through their common constant value $\ein_3$. The resulting map has degree $q$, and its energy tends to $-\infty$.
\end{proof}

\subsection{The endpoint topological barrier and right continuity}

Thus the degree-zero sequences in Proposition \ref{prop:weyl} may have arbitrarily small endpoint energy, whereas every nonzero-degree configuration has positive energy. Despite the loss of compactness caused by these sequences, the degree-$-1$ infimum remains continuous as the threshold is approached from the coercive side.

\begin{proposition}\label{prop:right-continuous}
The endpoint infimum is finite and
\[
\lim_{H \downarrow 1/4} I_{-1}(H)=I_{-1}(1/4) \ge\frac{b_*}{2}>0
\]
\end{proposition}

\begin{proof}
The construction leading to \eqref{eq:trial-energy} provides a comparison map $\bw_{\rho_*} \in \cA_{-1}$ with finite $\cF_2$-energy. Consequently,
\[
\widehat I_{-1}(2)\le \cF_2(\bw_{\rho_*}) <\infty.
\]
Applying the scaling identity \eqref{eq:exact} at $H=1/4$, where $\eps=2$, gives 
\[
I_{-1}(1/4)=\frac{1}{2}\widehat{I}_{-1}(2)<\infty.
\]
Now let $\m\in\cA$ satisfy $Q(\m)\neq0$, and define $\n(y)=\m(\sqrt{2}\,y)$.
Then $\n \in \cA$, $Q(\n)=Q(\m)$ and \eqref{eq:exact} gives
\[
E_{1/4}(\m)=\frac{1}{2}\cF_2(\n).
\]
Lemma~\ref{la:south-cost} therefore gives
\begin{equation}\label{eq:endpoint-bound}
E_{1/4}(\m)
\ge \frac{ b_*}{2}
=\frac{4\pi\alpha^2}{1-\alpha^2}>0.
\end{equation}
For every $H>1/4$ and every $\m\in\cA_{-1}$,
\[
E_H(\m) = E_{1/4}(\m) +\bra*{H-\frac{1}{4}}  \int_{\R^2}(1-m_3)\dif x,
\]
and hence
\[
I_{-1}(H)\ge I_{-1}(1/4).
\]
Conversely, for any $\eta>0$, choose $\m_\eta\in\cA_{-1}$ such that
\[
E_{1/4}(\m_\eta)<I_{-1}(1/4)+\eta.
\]
For $H>1/4$,
\[
I_{-1}(H) \le E_H(\m_\eta) = E_{1/4}(\m_\eta) +\bra*{H-\frac14} \int_{\R^2}(1-m_{\eta,3})\dif x.
\]
Therefore,
\[
\limsup_{H\downarrow1/4}I_{-1}(H) \le E_{1/4}(\m_\eta) <I_{-1}(1/4)+\eta.
\]
Letting $\eta\downarrow0$ proves
\[
\lim_{H\downarrow1/4}I_{-1}(H)=I_{-1}(1/4).
\]
\end{proof}

\begin{proof}[Proof of Theorem~\ref{thm:threshold}]
Part (i) is Proposition~\ref{prop:below-threshold}, part (ii) follows from \eqref{eq:endpoint-bound},  art (iii) is Proposition~\ref{prop:weyl} and and part (iv) follows from Proposition~\ref{prop:right-continuous}. 
\end{proof}

\section{The strong-field asymptotics}\label{sec:strong-field}

We now turn to the strong-field regime $H \to \infty$. Under the degree-preserving dilation $\m(x)=\n(H^{1/4}x)$, we have
\[
E_H(\m) = \sqrt{H}\cF_\eps(\n), \quad \cF_\eps(\n)=\cF_0(\n)-\eps D(\n), \quad \eps=H^{-1/2},
\]
where
\[
D(\n)=\frac12\int_{\R^2}|\nabla\n|^2\dif x.
\]
Recall that $\cG$ denotes the set of degree-$-1$ minimizers of $\cF_0$. By Theorem~\ref{thm:existence}, $\cG$ is nonempty and compact in $H^2$ modulo translations.

The identity $\cF_\eps=\cF_0-\eps D$ shows that the rescaled problem is a small perturbation of the limiting functional $\cF_0$. The comparison \eqref{eq:F0-comparison} turns minimizers of the perturbed problems into minimizing sequences for $\cF_0$, to which the strong $H^2$-compactness modulo translations established in Theorem~\ref{thm:existence} applies. These yield the leading-order limit, first-order profile selection, and the concentration and core-size laws in the original variables.

\subsection{Leading-order asymptotics and compactness}

We first identify the leading-order minimum and the limiting profiles of the corresponding rescaled minimizers.

\begin{proposition}\label{prop:leading}
As $\eps \downarrow 0$,
\begin{equation}\label{eq:Ihat-leading}
\widehat{I}_{-1}(\eps) \longrightarrow \widehat{I}_{-1}(0).
\end{equation}
Consequently, as $H \to \infty$,
\begin{equation}\label{eq:I-leading}
\frac{I_{-1}(H)}{\sqrt{H}}  \longrightarrow \widehat{I}_{-1}(0).
\end{equation}
\end{proposition}

\begin{proof}
Taking infimum in \eqref{eq:F0-comparison} over $\cA_{-1}$ gives
\[
\bra*{1-\frac{\eps}{2}} \widehat{I}_{-1}(0) \le \widehat{I}_{-1}(\eps) \le \widehat{I}_{-1}(0).
\]
Letting $\eps\downarrow0$ gives \eqref{eq:Ihat-leading}, while \eqref{eq:I-leading} follows from \eqref{eq:exact} with $\eps=H^{-1/2}$.
\end{proof}

Energy convergence alone does not control the minimizing profiles because of translation invariance. The comparison \eqref{eq:F0-comparison} instead turns $\cF_{\eps_k}$-minimizers into an $\cF_0$-minimizing sequence, yielding strong compactness modulo translations.

\begin{proposition}\label{prop:profile-compactness}
Let $\eps_k \downarrow 0$ and let $\n_k \in \cA_{-1}$ minimize $\widehat{I}_{-1}(\eps_k)$. Then
\begin{equation}\label{eq:nk-conv}
\cF_0(\n_k) \to \widehat{I}_{-1}(0)
\end{equation}
Consequently, after passing to a subsequence, there exist translations $y_k \in \R^2$ and $\n_* \in \cG$ such that
\begin{equation}\label{eq:profile-H2}
\n_k(\cdot + y_k) - \ein_3 \longrightarrow \n_* - \ein_3 \quad\text{strongly in } H^2.
\end{equation}
Moreover, the full sequence satisfies
\begin{equation}\label{eq:distance-G}
 \inf_{\substack{y\in\R^2,\,\bs g\in\cG}}  \norm{\n_k(\,\cdot+y)-\bs g}_{H^2}\longrightarrow 0.
\end{equation}
\end{proposition}

\begin{proof}

For all sufficiently large $k$, we have $\eps_k<2$. Since $\n_k\in\cA_{-1}$, the comparison \eqref{eq:F0-comparison} and the minimality of $\n_k$ give
\begin{align*}
\widehat I_{-1}(0) \le \cF_0(\n_k)
\le\frac{\cF_{\eps_k}(\n_k)}{1-\eps_k/2}=\frac{\widehat I_{-1}(\eps_k)}{1-\eps_k/2}
\longrightarrow \widehat I_{-1}(0),
\end{align*}
where the last convergence follows from Proposition $\ref{prop:leading}$. Hence \eqref{eq:nk-conv}, so $(\n_k)_{k\in\N}$ is a minimizing sequence for $\cF_0$. The strong $H^2$-compactness modulo translations given by Theorem $\ref{thm:existence}$ at $\eps=0$ therefore yields \eqref{eq:profile-H2}.

We prove \eqref{eq:distance-G} by contradiction. If it failed, there would exist $\delta>0$ and a subsequence $(k_j)_{j\in\N}$ such that
\[
\inf_{\substack{y\in\R^2,\, \bs g\in\cG}} \norm{\n_{k_j}(\,\cdot+y)-\bs g}_{H^2} \geq\delta
\]
for every $j\in\N$. Since $(\n_{k_j})_{j\in\N}$ is still a minimizing sequence for $\cF_0$, Theorem~\ref{thm:existence} gives, after passing to a further subsequence, translations $z_j\in\R^2$ and a profile $\bs g_*\in\cG$ such that
\[
\n_{k_j}(\,\cdot+z_j)-\ein_3 \longrightarrow
\bs g_*-\ein_3 \quad\text{strongly in }H^2(\R^2;\R^3).
\]
Consequently,
\[
\inf_{\substack{y\in\R^2,\, \bs g\in\cG}} \norm{\n_{k_j}(\,\cdot+y)-\bs g}_{H^2}
\le \norm{\n_{k_j}(\,\cdot+z_j)-\bs g_*}_{H^2} \longrightarrow 0,
\]
contradicting the assumed lower bound $\delta$. This proves \eqref{eq:distance-G}.
\end{proof}

\subsection{First-order expansion and profile selection}\label{sec:first-order}

Although the rescaled minimizers approach $\cG$, the leading-order limit does not determine which elements of $\cG$ are selected. Since $\cF_\eps=\cF_0-\eps D$, the first-order correction favors profiles with larger Dirichlet energy. We first show that $D$ attains its maximum on $\cG$.

\begin{lemma}\label{lem:Gstar}
There exists $\bs g_* \in \cG$ such that
\[
D(\bs g_*)= D_*= \sup_{\bs g \in \cG} D(\bs g).
\]
\end{lemma}

\begin{proof}
It follows from the interpolation \eqref{eq:interpolation} that,  for every $\bs g\in\cG$, 
\[
D(\bs g) =\frac{1}{2}\|\nabla\bs g\|_{L^2(\R^2)}^2
\le \frac{1}{2}\cF_0(\bs g) =\frac{1}{2}\widehat I_{-1}(0).
\]
Therefore, $0\le D_*\le\frac{1}{2}\widehat I_{-1}(0)<\infty$.
Choose $(\bs g_n) \subset \cG$ with $D(\bs g_n) \to \sup_{\cG} D$. Since $(\bs g_n)$ is a minimizing sequence for $\cF_0$, Theorem~\ref{thm:existence} gives translations $y_n\in\R^2$, a subsequence, and some $\bs g_*\in\cG$ such that
\[
\bs g_n(\cdot+y_n) - \ein_3 \longrightarrow \bs g_* - \ein_3 \quad\text{strongly in }H^2(\R^2;\R^3).
\]
Since $D$ is translation invariant and continuous under strong $H^2$-convergence, $D(\bs g_n) \to D(\bs g_*) = \sup_{\cG} D$.
\end{proof}

Consequently, the selected set $\cG_*=\set*{\bs g \in \cG, D(\bs g)=D_*}$ is nonempty.
We now show that every $\cF_\eps$ minimizer asymptotically approaches $\cG_*$ and derive the corresponding first-order energy expansion.

\begin{theorem}\label{thm:first-order}
As $\eps \downarrow 0$,
\begin{equation}\label{eq:Ihat-first}
\widehat{I}_{-1}(\eps)=\widehat{I}_{-1}(0)-\eps D_* + o(\eps).
\end{equation}
For each $\eps\in(0,2)$, let $\n_\eps$ be any minimizer of $\widehat I_{-1}(\eps)$. Then
\begin{equation}\label{eq:D-select}
D(\n_\eps) \longrightarrow D_*,
\end{equation}
\begin{equation}\label{eq:F0-rate}
\cF_0(\n_\eps)-\widehat{I}_{-1}(0) = o(\eps),
\end{equation}
and
\begin{equation}\label{eq:distance-Gstar}
\inf_{\substack{y\in\R^2 ,\, \bs g\in\cG_*}} \norm*{\n_\eps(\cdot+y)-\bs g}_{H^2} \longrightarrow 0.
\end{equation}
\end{theorem}

\begin{proof}
Let $\bs g_* \in \cG_*$. The minimality of $\n_\eps$ gives
\[
\cF_0(\n_\eps) -\eps D(\n_\eps) = \cF_\eps(\n_\eps) = \widehat{I}_{-1}(\eps) \le  \cF_\eps(\bs g_*) = \widehat{I}_{-1}(0)-\eps D_*.
\]
Since $\cF_0(\n_\eps) \ge \widehat{I}_{-1}(0)$,
\begin{equation}\label{eq:key-selection}
0 \le \cF_0(\n_\eps) - \widehat{I}_{-1}(0) \le \eps (D(\n_\eps)-D_*).
\end{equation}
Equation \eqref{eq:key-selection} first implies that $D(\n_\eps)\geq D_*$. We claim that $D(\n_\eps)\to D_*$. Otherwise, there would exist $\delta>0$ and a sequence $\eps_k\downarrow0$ such that $D(\n_{\eps_k})\geq D_*+\delta$ for every $k$. By Proposition \ref{prop:profile-compactness}, after passing to a subsequence, there exist translations $y_k\in\R^2$ and some $\bs g\in\cG$ such that
\[
\n_{\eps_k}(\cdot+y_k)-\ein_3 \longrightarrow \bs g-\ein_3 \quad\text{strongly in }H^2.
\]
Since $D$ is translation invariant and continuous under strong $H^2$-convergence, it follows that
\[
D(\n_{\eps_k})\longrightarrow D(\bs g)\le D_*,
\]
which is a contradiction. This proves \eqref{eq:D-select}.

Moreover, if a translated subsequence of minimizers converges strongly in $H^2$ to some $\bs g\in\cG$, then \eqref{eq:D-select} and the continuity of $D$ give $D(\bs g)=D_*$.
Thus every such limit belongs to $\cG_*$. Finally, \eqref{eq:key-selection} and \eqref{eq:D-select} yield
\[
0\le \cF_0(\n_\eps)-\widehat I_{-1}(0) \le \eps\bigl(D(\n_\eps)-D_*\bigr) =o(\eps),
\]
which proves \eqref{eq:F0-rate}. Consequently,
\[
\widehat I_{-1}(\eps) =\cF_0(\n_\eps)-\eps D(\n_\eps) =\widehat I_{-1}(0)-\eps D_*+o(\eps),
\]
and hence \eqref{eq:Ihat-first} follows.

It remains to prove \eqref{eq:distance-Gstar}. If it failed, there would exist $\delta>0$, a sequence $\eps_k\downarrow0$, and corresponding minimizers $\n_{\eps_k}$ such that
\[
\inf_{\substack{y\in\R^2,\, \bs g\in\cG_*}} \norm{\n_{\eps_k}(\,\cdot+y)-\bs g}_{H^2} \geq\delta
\]
for every $k$.
By Proposition \ref{prop:profile-compactness}, after translations and extraction, $\n_{\eps_k}$ converges strongly in $H^2$ to some $\bs g\in\cG$. The preceding selection argument shows that every such limit belongs to $\cG_*$, giving a contradiction.
\end{proof}

We now translate these conclusions back to the original variables and prove Theorem~\ref{thm:strong-field}.
\begin{proof}[Proof of Theorem~\ref{thm:strong-field}]
The maximum $D_*$ is attained by Lemma \ref{lem:Gstar}. Taking
$\eps=H^{-1/2}$ in \eqref{eq:Ihat-first} and using \eqref{eq:exact}
gives
\[
I_{-1}(H)
=
\sqrt{H}\widehat{I}_{-1}(0)-D_*+o(1).
\]

To prove the profile statement, let $\m_H$ be a minimizer of
$I_{-1}(H)$ and define
\[
\n_H(z):=\m_H(H^{-1/4}z).
\]
By \eqref{eq:exact}, $\n_H$ minimizes
$\widehat{I}_{-1}(H^{-1/2})$. Hence
\eqref{eq:distance-Gstar} gives
\[
\inf_{\substack{y\in\R^2,\, \bs g\in\cG_*}}
\norm*{\n_H(\,\cdot+y)-\bs g}_{H^2}
\longrightarrow 0.
\]
For every $y\in\R^2$, setting $a=H^{-1/4}y$ yields
\[
\n_H(\,\cdot+y)
=
\m_H\bigl(a+H^{-1/4}\,\cdot\bigr).
\]
Since $y\mapsto a=H^{-1/4}y$ is a bijection of $\R^2$, the preceding
convergence is equivalent to
\[
\inf_{\substack{a\in\R^2,\, \bs g\in\cG_*}}
\norm*{\m_H(a+H^{-1/4}\,\cdot)-\bs g}_{H^2}
\longrightarrow 0.
\]
This proves the profile assertion.
\end{proof}

\subsection{Pohozaev balance, concentration, and the core scale}

First we apply Derrick's scaling argument \cite{Derrick1964} to obtain a Pohozaev identity. If $\m_H$ minimizes $E_H$ in $\cA_q$ for $q \in \Z$, then for the degree-preserving family $\m_{H,\lambda}(x)=\m_H(x/\lambda)$, we have
\[
0 = \left.\frac{\dif}{\dif \lambda} E_H(\m_{H,\lambda})\right|_{\lambda=1}
= -\int_{\R^2} |\Delta \m_H|^2 \dif x + 2H \int_{\R^2} (1-m_{H,3}) \dif x.
\]
Hence
\begin{equation}\label{eq:Pohozaev}
\int_{\R^2} |\Delta \m_H|^2 \dif x = 2H \int_{\R^2} (1-m_{H,3}) \dif x.
\end{equation}
We combine this exact identity with the profile-selection result. Let $H_k\to\infty$, let $\m_{H_k}$ be a minimizer of $I_{-1}(H_k)$, and set $\delta_k=H_k^{-1/4}$. After passing to a subsequence, not relabeled, there exist $a_k\in\R^2$ and $\n_*\in\cG_*$ such that
\[
\n_k(y):=\m_{H_k}(a_k+\delta_k y)
\]
satisfies
\begin{equation}\label{eq:uk-definition}
\n_k(y) - \ein_3 \longrightarrow \n_* - \ein_3  \quad\text{strongly in } H^2.
\end{equation}
After the change of variables $x=a_k+\delta_k y$, \eqref{eq:Pohozaev} becomes
\[
\int_{\R^2} |\Delta \n_k|^2 \dif y = 2 \int_{\R^2} (1-n_{k,3}) \dif y.
\]
Strong convergence in $H^2$ allows us to pass to the limit and obtain
\[
\int_{\R^2} |\Delta \n_*|^2 \dif y = 2 \int_{\R^2} (1-n_{*,3}) \dif y.
\]
Combining this with $\cF_0(\n_*) = \widehat{I}_{-1}(0)$ gives
\begin{equation}\label{eq:equipartition}
\int_{\R^2}\frac{1}{2} |\Delta \n_*|^2 \dif y =  
 \int_{\R^2} (1-n_{*,3}) \dif y = \frac{\widehat{I}_{-1}(0)}{2}.
\end{equation}
We next return to the original variables and formulate the concentration limits. For a finite signed Radon measure $\nu$, recall the translation notation
\[
\mathscr{T}_a\nu=(x\mapsto x-a)_\#\nu.
\]
The leading energy is of order $\sqrt{H_k}$, as shown in Proposition \ref{prop:leading}.  At this scale, the negative Dirichlet contribution is lower order, so we distinguish the nonnegative leading-order energy from the full signed energy. The corresponding concentration limits are collected in the following proposition.

\begin{proposition}\label{prop:charge-energy-concentration}
Let $H_k\to\infty$, let $\m_{H_k}$ minimize $I_{-1}(H_k)$. Suppose that there exist centers $a_k\in\R^2$ and a profile $\n_*\in\cG_*$ such that the rescaled maps
\[
\n_k(y):=\m_{H_k}(a_k+H_k^{-1/4}y)
\]
satisfy
\[
\n_k-\ein_3\longrightarrow\n_*-\ein_3 \qquad\text{strongly in }H^2(\R^2;\R^3).
\]
Then the following weak-$*$ limits hold.

The topological-charge measure concentrates with total weight $-4\pi$,
\begin{equation}\label{eq:charge-concentration}
\mathscr{T}_{a_k} \bra*{q(\m_{H_k}) \dif x} \stackrel{*}{\rightharpoonup} -4\pi\delta_0,
\end{equation}
while the normalized leading-order energy concentrates with total weight $\widehat{I}_{-1}(0)$,
\begin{equation}\label{eq:positive-concentration}
\mathscr{T}_{a_k} \bra*{\frac{1}{\sqrt{H_k}} \pra*{\frac{1}{2}|\Delta \m_{H_k}(x)|^2  +H_k(1-m_{H_k,3}(x))}\dif x} \stackrel{*}{\rightharpoonup} \widehat{I}_{-1}(0)\delta_0,
\end{equation}
and its fourth-order and potential components each contribute one half of this value:
\begin{equation}\label{eq:fourth-concentration}
\mathscr{T}_{a_k} \bra*{\frac{|\Delta \m_{H_k}(x)|^2}{2\sqrt{H_k}}\dif x} \stackrel{*}{\rightharpoonup} \frac{\widehat{I}_{-1}(0)}{2}\delta_0,
\end{equation}
\begin{equation}\label{eq:potential-concentration}
\mathscr{T}_{a_k} \bra*{\sqrt{H_k}(1-m_{H_k,3})\dif x} \stackrel{*}{\rightharpoonup} \frac{\widehat{I}_{-1}(0)}{2}\delta_0.
\end{equation}
Moreover, the full energy measure has the same limit as its leading-order part:
\begin{equation}\label{eq:full-concentration}
\mathscr{T}_{a_k} \bra*{\frac{1}{\sqrt{H_k}} \pra*{\frac{1}{2}|\Delta \m_{H_k}(x)|^2 -\frac{1}{2}|\nabla \m_{H_k}|^2 +H_k(1-m_{H_k,3}(x))}\dif x} \stackrel{*}{\rightharpoonup} \widehat{I}_{-1}(0)\delta_0.
\end{equation}
\end{proposition}

\begin{proof}
The proof relies on the following elementary concentration principle. If $f_k \to f$ strongly in $L^1(\R^2)$ and $\delta_k \to 0$, then
\begin{equation}\label{eq:L1-collapse}
\int_{\R^2} \varphi(\delta_k y) f_k(y) \dif y \longrightarrow \varphi(0) \int_{\R^2}  f(y) \dif y,
\end{equation}
for any test function $\varphi \in C_0(\R^2)$.
Indeed,
\begin{align*}
&\abs*{ \int_{\R^2}\varphi(\delta_ky)f_k(y)\dif y  -\varphi(0)\int_{\R^2}f(y)\dif y}\\
\quad\le &\norm{\varphi}_{L^\infty(\R^2)}\norm{f_k-f}_{L^1(\R^2)}
 +\int_{\R^2}
   |\varphi(\delta_ky)-\varphi(0)|\,|f(y)|\dif y
 \longrightarrow0,
\end{align*}
where the second term converges to zero by dominated convergence.

For the charge density, strong $H^2$-convergence in \eqref{eq:uk-definition} implies that
\begin{align*}
&  \norm{q(\n_k)-q(\n_*)}_{L^1} \le \norm{\n_k - \n_*}_{L^\infty} \norm{\partial_1\n_k}_{L^2(\R^2)}\norm{\partial_2 \n_k}_{L^2}\\
& \quad + \norm{\partial_1 (\n_k-\n_*)}_{L^2} \norm{\partial_2 \n_k}_{L^2} + \norm{\partial_1 \n_*}_{L^2} \norm{\partial_2 (\n_k - \n_*)}_{L^2} \longrightarrow 0.
\end{align*}
The change of variables $x=a_k + \delta_k y$ gives
\[
\int_{\R^2} \varphi(x-a_k) q(\m_{H_k})(x) \dif x = \int_{\R^2} \varphi(\delta_k y) q(\n_k)(y) \dif y.
\]
Applying \eqref{eq:L1-collapse} with $f_k=q(\n_k)$ and $f=q(\n_*)$, and using $\int_{\R^2} q(\n_*)=-4\pi$ yields \eqref{eq:charge-concentration}.
For the energy components, set
\[
h_k(y)=\frac{1}{2}|\Delta \n_k(y)|^2+1-n_{k,3}(y), \quad
h_*(y)=\frac{1}{2}|\Delta \n_*(y)|^2+1-n_{*,3}(y).
\]
Strong convergence in $H^2$  and $1-n_3 = |\n-\ein_3|^2/2$ imply
\[
\frac{1}{2}|\Delta \n_k|^2 \to \frac{1}{2}|\Delta \n_*|^2, \quad 1-n_{k,3} \to 1-n_{*,3} \quad\text{strongly in } L^1(\R^2).
\]
Thus $h_k \to h_*$ strongly in $L^1(\R^2)$. Hence
\begin{align*}
& \int_{\R^2} \varphi(x-a_k) \frac{1}{\sqrt{H_k}} \pra*{\frac{1}{2}|\Delta \m_{H_k}(x)|^2  +H_k(1-m_{H_k,3}(x))} \dif x \\
= & \int_{\R^2} \varphi(\delta_k y) h_k(y) \dif y \longrightarrow \varphi(0)  \int_{\R^2} h_*(y) \dif y = \varphi(0) \widehat{I}_{-1}(0).
\end{align*}
Applying the same argument separately to $|\Delta \n_k|^2/2$ and $1-n_{k,3}$, and using \eqref{eq:equipartition}, gives \eqref{eq:fourth-concentration} and \eqref{eq:potential-concentration}. Finally, for the Dirichlet contribution we have
\[
\frac{1}{2\sqrt{H_k}} \int_{\R^2}  |\nabla \m_{H_k}|^2 \dif x = \frac{D(\n_k)}{\sqrt{H_k}} \longrightarrow 0,
\]
since $D(\n_k)$ remains bounded by \eqref{eq:uk-definition}. Combining this with \eqref{eq:positive-concentration} yields \eqref{eq:full-concentration}.
\end{proof}

The preceding weak-$*$ limits show that the charge and energy collapse to $a_k$, but they do not by themselves quantify the core radius. Following the concentration-function approach in \cite{Lions1984}, we use the potential-energy concentration radius introduced in Theorem~\ref{thm:concentration}.
For the minimizing sequence $(\m_{H_k})$ and the centers $a_k$ appearing in \eqref{eq:uk-definition}, recall that
\[
R_{k,\lambda} = \inf \set*{R>0: \int_{B_R(a_k)} (1-m_{H_{k},3})\dif x \ge \lambda \int_{\R^2} (1-m_{H_{k},3}) \dif x}.
\]
The following corollary shows that $R_{k,\lambda}$ is comparable to $H_k^{-1/4}$ for every fixed $\lambda\in(0,1)$.
\begin{corollary}\label{cor:core-scale}
Under the assumptions and notation of Proposition~\ref{prop:charge-energy-concentration}, for each $\lambda \in (0,1)$, there are constants $0<c_\lambda<C_\lambda<\infty$ such that
\[
c_\lambda H^{-1/4}_k \le R_{k,\lambda} \le C_\lambda H^{-1/4}_k
\]
for sufficiently large $k$. 
\end{corollary}
\begin{proof}
Define
\[
P_k = \int_{\R^2} 1-n_{k,3}(y) \dif y, \qquad
P_* = \int_{\R^2} 1-n_{*,3}(y) \dif y.
\]
Since $1-n_{k,3} \to 1-n_{*,3}$ strongly in $L^1(\R^2)$ and $Q(\n_*)=-1$, we have $P_k \to P_*$ and $P_*>0$.  Hence we can choose $c_\lambda, C_\lambda>0$ so that
\[
\int_{B_{c_\lambda}} (1-n_{*,3}) \dif y < \lambda P_*, \quad
\int_{B_{C_\lambda}} (1-n_{*,3}) \dif y > \lambda P_*.
\]
The strong $L^1$-convergence of $1-n_{k,3}$ and $P_k \to P_*$ preserve these strict inequalities  for sufficiently large $k$:
\[
\int_{B_{c_\lambda}} (1-n_{k,3}) \dif y < \lambda P_k, \quad
\int_{B_{C_\lambda}} (1-n_{k,3}) \dif y > \lambda P_k.
\]
Returning to the original variable $x=a_k + \delta_k y$ gives
\[
\frac{\int_{B_{r\delta_k}(a_k)} (1-m_{H_{k},3}) \dif x}{\int_{\R^2} (1-m_{H_{k},3}) \dif x}
= \frac{1}{P_k} \int_{B_r} (1-n_{k,3}) \dif y \quad\text{for every } r>0.
\]
Taking $r=c_\lambda,C_\lambda$ proves the result.
\end{proof}

\begin{proof}[Proof of Theorem~\ref{thm:concentration}]
The subsequential profile statement follows from Proposition \ref{prop:profile-compactness} and Theorem~\ref{thm:first-order}.
The remaining measure-concentration and core-scale assertions are precisely Proposition \ref{prop:charge-energy-concentration} and Corollary \ref{cor:core-scale}.
\end{proof}

\bibliographystyle{abbrv}
\bibliography{references}

\end{document}